\documentclass[11pt]{amsart}

\usepackage{amsfonts}
\usepackage{amssymb}
\usepackage{amsmath}
\usepackage{amsthm}

\usepackage{graphicx}
\usepackage{xcolor}

\usepackage{hyperref}

\theoremstyle{definition}

\newtheorem{remark}{Remark}
\newtheorem*{claim}{Claim}
\theoremstyle{plain}
\newtheorem{theorem}{Theorem}
\newtheorem{lemma}[theorem]{Lemma}

\newtheorem*{theorem*}{Theorem}
\newtheorem*{conjecture}{Conjecture}

\begin{document}
\title[growth of the number of periodic orbits]{The growth of the number of periodic orbits for annulus homeomorphisms and non-contractible closed geodesics on Riemannian or Finsler $\mathbb{R}P^2$}

\author{Hui Liu}
\address{School of Mathematics and Statistics, Wuhan University,
Wuhan 430072, Hubei, P. R. China}
\email{huiliu00031514@whu.edu.cn}

\author{Jian Wang}
\address{School of Mathematical Sciences and LPMC, Nankai University, Tianjin 300071, P. R. China}
\email{wangjian@nankai.edu.cn}

\author{Jingzhi Yan}
\address{College of Mathematics, Sichuan University, Chengdu 610065, P.R. China }
\email{jyan@scu.edu.cn}

\date{}

\begin{abstract}
In this article, we give a growth rate about the number of periodic orbits in the Franks type theorem obtained by the authors \cite{LWY}. As applications, we prove the following two results: there exist infinitely many distinct non-contractible closed geodesics on $\mathbb{R}P^2$ endowed with a Riemannian metric such that its Gaussian curvature is positive, moreover, the number of non-contractible closed geodesics of length $\leq l$ grows at least like $l^2$; and there exist either two or infinitely many distinct non-contractible closed geodesics on Finsler $\mathbb{R}P^2$ with reversibility $\lambda$ and flag curvature $K$ satisfying $\left(\frac{\lambda}{1+\lambda}\right)^2<K\le 1$, furthermore, if the second case happens, then the number of non-contractible closed geodesics of length $\leq l$ grows at least like $l^2$.
\end{abstract}

\maketitle

\section{introduction}

The research of the periodic orbits of annulus homeomorphisms was started by Poincar\'{e}. In his search for periodic solutions in the restricted three body problem of celestial mechanics, Poincar\'{e} constructed an area-preserving section map of an annulus $\mathbb{A}\backsimeq\mathbb{R}/\mathbb{Z}\times [0,1]$ on the energy surface, and get the periodic orbits of the original system from the periodic orbits of the annulus homeomorphism (see \cite{Poincare}). Later, many mathematicians studied annulus homeomorphisms (see \cite{Birkhoff13, Birkhoff15, Birkhoff26, Franks92, Lecalvez01, Lecalvez05} etc.)

Poincar\'{e}'s idea inspired many mathematicians. A classic method for the existence of periodic motions of Hamiltonian system with two degrees of freedom is as follows: we first reduce the dynamics to the annulus-type global surface of section and then get the existence of two or infinitely many periodic orbits by the following theorem due to Franks; see the pioneer works of Hofer, Wysocki and Zehnder \cite{HWZ98, HWZ03} and also some important progress~\cite{CGHP, CDR22} on this topic.

\begin{theorem*}[Franks]\label{thm:Franks} Suppose that $f$ is an area preserving homeomorphism of the open or closed annulus which is isotopic to the identity. If $f$ has at least one fixed or periodic point, then $f$ must have infinitely many interior periodic points.
\end{theorem*}

In \cite{Franks92}, Franks used this theorem to prove that a Riemannian metric on $\mathbb{S}^2$ has infinitely many closed geodesics whenever there is a simple closed geodesic for which  Birkhoff's global annulus like surface of section is well defined. This, together with Bangert's result \cite{Bangert1993}, implies the existence of infinitely many closed geodesics on every Riemannian $\mathbb{S}^2$.

In our last paper \cite{LWY}, we improved Franks' theorem by considering the existence of the periodic orbits with periods relatively prime to a given number $n_0$, and as its application we gave some precise information about the symmetries of periodic orbits found in Hofer, Wysocki and Zehnder's dichotomy theorem when the tight 3-sphere is equipped with some additional symmetries; see also \cite{FK, Sch20, Kim} for other dichotomy results about symmetric periodic orbits.

For the growth of the number of periodic orbits, Neumann \cite{Neumann77}  considered  area preserving twist homeomorphism of the closed annulus. In this article,  we will  estimate  the growth of the number of periodic orbits  and the growth of the number of periodic orbits whose periods are prime to a given prime number for area preserving homeomorphisms of the open and closed annuli,  which  will be applied to the problem of non-contractible closed geodesics on Finsler or Riemannian $\mathbb{R}P^2$.

\subsection{The growth of the number of periodic orbits for annulus homeomorphisms}

We consider a homeomorphism $f$ of the open (resp. closed) annulus $\mathbb{A}=\mathbb{R}/\mathbb{Z}\times (0,1)$ (resp. $\mathbb{A}=\mathbb{R}/\mathbb{Z}\times [0,1]$) that is isotopic to the identity.
An \emph{area} of a surface  is  a locally finite Borel  measure without atom and with total support.
We say that $z$ is an {\it $n$ prime-periodic point} of $f$ if $z$ is an $n$ periodic  point but not an $l$ periodic point of $f$ for all $0<l<n$. We call a periodic orbit of $f$ {\it an interior} periodic orbit, if it is in the interior of $\mathbb{A}$.
We denote by $\mathrm{Fix}(f)$ the set of fixed points of $f$.
Let
\[N_{=n}=\sharp\{\text{interior $n$ prime-periodic orbits of } f \},\]
\[N_{\le n}=\sharp\{\text{interior periodic orbits of $f$ with prime-period}\le n \},\]
\[N_{\le n, n_0\nmid}=\sharp\{\text{interior } q \text{ prime-periodic orbits of }  f:  q\le n, (q, n_0)=1 \}.\]
We have the following theorem about the growth  of the  number of the  periodic orbits of $f$.

\begin{theorem}\label{thm: growth of periodic orbits}
  Let $f$ be a homeomorphism of the closed or open annulus, that is isotopic to the identity and preserves a finite area. If $f$ has a fixed or  periodic point, then
   \[\liminf_{n\to+\infty} N_{\le n} \frac{1}{n^2}>0.\]
   Moreover, for  a given prime number $n_0$, if $f$ has a fixed or periodic point with prime-period $k$ such that $(k, n_0)=1$, then
   \[\liminf_{n\to+\infty} N_{\le n, n_0\nmid}  \frac{1}{n^2}>0.\]
    More precisely, if  $f$ has  a $k$ prime-periodic point  and  $f^k\not= \mathrm{Id}$, then
  \[\liminf_{n\to +\infty} N_{=kn}\frac{\log\log n}{n}>0.\]
\end{theorem}

As a corollary of the second part of the theorem, if $f$ has an odd periodic orbit, it has infinite many interior odd periodic orbits, and the number of interior odd periodic orbits with prime-periods not exceeding $n$ grows at least like $n^2$. We also note that the third part of the theorem  is a direct corollary of  Theorem \ref{thm: growth of n periodic} in Section 3, which is more precise but  technical.

\begin{remark}
 Neumann \cite{Neumann77} gave an example of an area preserving  homeomorphism of the closed annulus,  such that $ N_{=n} \le (2+\varepsilon)\varphi(n; a, b)$, where $\varepsilon$ is small,  $a, b$ are the different rotation numbers of the points on the two boundaries, and  $\varphi(n; a, b)$ is the number of irreducible  fractions with denominator $n$ in the interval $(a, b)$. By choosing $a<0<b$, we have an area preserving  homeomorphism of the closed annulus that has a fixed point and satisfies
 \begin{align*}
 &\limsup_{n\to+\infty} N_{\le n} \frac{1}{n^2}<+\infty,\\
  &\limsup_{n\to+\infty} N_{\le n, n_0\nmid}  \frac{1}{n^2}<+\infty,\\
  &\liminf\limits_{n\to +\infty} N_{=n}\frac{\log\log n}{n}<+\infty.
 \end{align*}
\end{remark}

\begin{remark}
  If we add the condition that $f$ is reversible and consider the number of  symmetric periodic orbits
(see  \cite{Kang} or \cite{LWY} for the precise definitions), we  have similar results.  The idea of the proof is similar, but we should add the condition that $\sharp\mathrm{Fix}(f^k)<+\infty$ when we consider the number of the $kn$ prime-periodic symmetric orbits.
\end{remark}

\subsection[The growth of the number of non-contractible closed geodesics on Riemannian or Finsler $\mathbb{R}P^2$]{\texorpdfstring{The growth of the number of non-contractible closed geodesics on Riemannian or Finsler $\mathbb{R}P^2$}{The growth of the number of non-contractible closed geodesics on Riemannian or Finsler RP2}}

Hingston \cite{Hingston1993} proved that the number of closed geodesics of length $\leq l$ on Riemannian $\mathbb{S}^2$ grows at least like the prime numbers. Similarly, we  consider the growth of the number of non-contractible closed geodesics on
Riemannian $\mathbb{R}P^2$ whose double cover is $\mathbb{S}^2$. As applications of our Theorem \ref{thm: growth of periodic orbits}, we obtain
the following:

\begin{theorem}\label{thm: application 1}
Let $(\mathbb{R}P^2, g)$ be a Riemannian  real projective plane whose Gaussian curvature is positive.  Then there exist infinitely many distinct non-contractible closed geodesics on $(\mathbb{R}P^2,g)$. Moreover, the number of non-contractible closed geodesics of length $\leq l$ grows at least like $l^2$.
\end{theorem}

For Finsler $\mathbb{R}P^2$ under suitable conditions,  a dichotomy result on the existence of non-contractible closed geodesics holds. Moreover, if there exist infinitely many distinct non-contractible closed geodesics, then the growth rate as in Theorem \ref{thm: application 1} holds:

\begin{theorem}\label{thm: application 2}
Let $(\mathbb{R}P^2,F)$ be an irreversible Finsler real projective plane with reversibility $\lambda$ and flag curvature $K$ satisfying $\left(\frac{\lambda}{1+\lambda}\right)^2<K\le 1$. Then there exist either two or infinitely many distinct non-contractible closed geodesics on $(\mathbb{R}P^2,F)$; furthermore, if the second case happens, then the number of non-contractible closed geodesics of length $\leq l$ grows at least like $l^2$.
\end{theorem}

\subsection{Organization of the paper}
In Section 2, we will give several definitions and preliminary results. In Section 3, we will prove  Theorem \ref{thm: growth of periodic orbits} and Theorem \ref{thm: growth of n periodic}.
In Section 4, we review the problem about the multiplicity of non-contractible closed geodesics on
$\mathbb{R}P^2$ endowed with Riemannian or Finsler metrics and then apply
Theorem \ref{thm: growth of periodic orbits} to prove Theorem~\ref{thm: application 1} and Theorem~\ref{thm: application 2}.

\subsection{Acknowledgements.}
Hui Liu was partially supported by NSFC (Nos. 12022111, 11771341).
Jian Wang was partially supported by NSFC (Nos. 12071231, 11971246) and the Fundamental Research Funds for the Central University. Jingzhi Yan was partially supported by NSFC(Nos. 11901409, 11831012).


\section{Preliminaries}

\subsection{Euler's phi function and the number of irreducible fractions}

Let $\varphi(n)$ be the Euler's phi function, i.e., $\varphi(n)$ is equal to the number of positive integers not exceeding $n$ and relatively prime to $n$. We have the following classic results.

\begin{theorem}\cite[Section 5.5, 18.4, 18.5]{HW}\label{thm: the order of euler's phi function}
\begin{itemize}
\item[i)] If $\gcd(m, n)=1$, then $\varphi(mn)=\varphi(m)\varphi(n)$.

\item[ii)] Let $n=p_1^{r_1}\cdots p_s^{r_s}$ be the prime decomposition of n. Then,
\[\varphi(n)=n\prod_{i=1}^{s}(1-\frac{1}{p_i})=\prod_{i=1}^{s}(p_i^{r_i}-p_i^{r_i-1}).\]

\item[iii)]
\[\liminf_{n\to +\infty}\varphi(n)\frac{\log\log n}{n}=e^{-\gamma},\]
where $\gamma=\lim\limits_{n\to +\infty}(1+\frac{1}{2}+\frac{1}{3}+\cdots+\frac{1}{n}-\log n)$ is the Euler's constant.

\item[iv)] $\Phi(n)\overset{\text{def}}{=}\varphi(1)+\cdots+\varphi(n)=\frac{3n^2}{\pi^2}+ O(n\log n)$,  as $n\to+\infty$.
\end{itemize}
\end{theorem}

\bigskip

Now, let us consider the irreducible fractions between two real number  $\rho^-<\rho^+$.
Denote by $\varphi(n; \rho^-,\rho^+)$  the number of irreducible  fractions with the denominator $n$ between $\rho^-<\rho^+$, i.e.,
\begin{align*}
\varphi(n; \rho^-, \rho^+)&=\sharp\{m\in\mathbb{Z}:  \rho^-<\frac{m}{n}<\rho^+, \gcd(m,n)=1\}\\
&=\sharp\{m\in\mathbb{Z}:  m\in(n\rho^-, n\rho^+),   \gcd(m,n)=1\}.
\end{align*}
Let $\Phi(n; \rho^-,\rho^+)$ be the number of irreducible  fractions with denominators not exceeding $n$ between $\rho^-<\rho^+$, i.e.,
\begin{align*}
\Phi(n; \rho^-, \rho^+)&=\sharp\{\frac{p}{q}: \rho^-<\frac{p}{q}<\rho^+,  q\le n,   \gcd(p,q)=1\}\\
&=\sum_{q=1}^{n}\varphi(q;  \rho^-, \rho^+).
\end{align*}
We have  the following estimations of the orders of $\varphi(n; \rho^-,\rho^+)$  and $\Phi(n; \rho^-,\rho^+)$.

\begin{lemma}\cite{Neumann77}\label{Lemma: order of phi(n; rho-, rho+)}
  $\varphi(n; \rho^-,\rho^+)\sim (\rho^+-\rho^-)\varphi(n)$, as $n\to +\infty$.
\end{lemma}

\begin{lemma}\label{lemma: order of Phi(n; rho-, rho+)}
  $\Phi(n; \rho^-, \rho^+)\sim \frac{3(\rho^+-\rho^-)}{\pi^2} n^2$, as $n\to +\infty$.
\end{lemma}

The following proof of Lemma \ref{Lemma: order of phi(n; rho-, rho+)} is from \cite{Neumann77}, and the proof of Lemma~\ref{lemma: order of Phi(n; rho-, rho+)} will use the result in the proof of Lemma~\ref{Lemma: order of phi(n; rho-, rho+)}.

\begin{proof}[Proof of Lemma \ref{Lemma: order of phi(n; rho-, rho+)}]
 Let $n=p_1^{r_1}\cdots p_s^{r_s}$ be the prime decomposition of $n$. Then, in the interval $(0,n]$, there are
 \begin{itemize}
   \item $n$  integers,
   \item  $\frac{n}{p_i}$  integers that are divisible by  $p_i$, for $i=1,\cdots, s$,
   \item $\frac{n}{p_i p_j}$  integers that are divisible by  $p_ip_j$, for $1\le i<j\le s$,
   \item $\ldots$
 \end{itemize}
 By the inclusion-exclusion principle,
 \[\varphi(n)=n-\sum_{i=1}^{s}\frac{n}{p_i}+\sum_{1\le i<j\le s}\frac{n}{p_i p_j}-\sum_{1\le i<j<k\le s}\frac{n}{p_i p_j p_k}+\cdots\]

For a given  positive integer $q$, in an open interval with length $\ell$,
 \begin{itemize}
 \item
 if $\frac{\ell}{q}$ is not an integer, there are   $\lfloor\frac{\ell}{q}\rfloor$ or $\lfloor\frac{\ell}{q}\rfloor+1$ integers that are divisible by $q$ in the interval;
 \item  if $\frac{\ell}{q}$ is  an integer, there are   $\frac{\ell}{q}$ or $\frac{\ell}{q}-1$ integers that are divisible by $q$ in the interval.
 \end{itemize}
So,
\begin{align*}
 &\varphi(n; \rho^-,\rho^+)\\
 =&[(n\rho^+-n\rho^-)-\sum_{i=1}^{s}\frac{n\rho^+-n\rho^-}{p_i}+\sum_{1\le i<j\le s}\frac{n\rho^+-n\rho^-}{p_i p_j}\\
 &-\sum_{1\le i<j<k\le s}\frac{n\rho^+-n\rho^-}{p_i p_j p_k}+\cdots]+\varepsilon(n)\\
 =&(\rho^+-\rho^-)\varphi(n)+\varepsilon(n),
\end{align*}
where $\varepsilon(n)$ is the error term, and
 \[|\varepsilon(n)|\le 1+\sum_{i=1}^{s}1+\sum_{1\le i<j\le s}1+\cdots =2^s. \]
Recall that $\varphi(n)=\prod_{i=1}^{s}(p_i^{r_i}-p_i^{r_i-1})$. So, $\frac{\varepsilon(n)}{\varphi(n)}\to 0$, as $n\to +\infty$.

Therefore, $\varphi(n; \rho^-,\rho^+)\sim (\rho^+ -\rho^-)\varphi(n)$, as $n\to +\infty$.
\end{proof}

\begin{proof}[Proof of Lemma \ref{lemma: order of Phi(n; rho-, rho+)}]
In the proof of the previous lemma, we have
\[\varphi(n; \rho^-, \rho^+)=(\rho^+-\rho^-)\varphi(n)+\varepsilon(n),\]
with $\lim\limits_{n\to+\infty}\frac{\varepsilon(n)}{\varphi(n)}= 0$.

So, $ \Phi(n; \rho^-, \rho^+)=\sum\limits_{q=1}^{n}\varphi(q; \rho^-, \rho^+)=(\rho^+-\rho^-)\sum\limits_{q=1}^{n}\varphi(q)+\sum\limits_{q=1}^{n}\varepsilon(q).$

Note that $\varphi(n)$, $n\in\mathbb{N}$ are positive integers.   Hence,
\[\lim_{n\to+\infty}\frac{\sum_{q=1}^{n}\varepsilon(q)}{\sum_{q=1}^{n}\varphi(q)}=0.\]

Therefore,
\[\Phi(n; \rho^-, \rho^+)\sim (\rho^+-\rho^-)\Phi(n)\sim \frac{3(\rho^+-\rho^-)}{\pi^2} n^2,\quad  \text{ as } n\to +\infty. \qedhere\]
\end{proof}

Now, for a given prime number $n_0$, we consider the number of irreducible  fractions  between $\rho^-<\rho^+$, whose denominators are relatively prime to $n_0$ and are not exceeding $n$. Let
\begin{align*}
\Psi(n; \rho^-, \rho^+)&=\sharp\{\frac{p}{q}:  \rho^-<\frac{p}{q}<\rho^+,  q\le n,   \gcd(p,q)=1, \gcd(q, n_0)=1.\}\\
&=\sum_{\begin{subarray}{l} q=1,2,\cdots, n;\\
 \gcd(q, n_0)=1  \end{subarray}}\varphi(q;  \rho^-, \rho^+).
\end{align*}
We will estimate  the order of  $\Psi(n; \rho^-,\rho^+)$.

\begin{lemma}\label{lemma: order of Psi(n; rho-, rho+)}
\[0<C_1\le\liminf_{n\to+\infty}\Psi(n; \rho^-, \rho^+)\frac{1}{n^2}\le \limsup_{n\to +\infty} \Psi(n; \rho^-, \rho^+)\frac{1}{n^2}\le C_2<+\infty, \]
  where $C_1$, $C_2$ are constants depending on  $\rho^+$, $\rho^-$ and the prime number $n_0$.
\end{lemma}

\begin{proof}
  \begin{align*}
\Psi(n; \rho^-, \rho^+)&=\sum_{\begin{subarray}{l} q=1,2,\cdots, n;\\
 \gcd(q, n_0)=1
 \end{subarray}}\left((\rho^+-\rho^-)\varphi(q)+\varepsilon(q)\right)\\
&=(\rho^+-\rho^-)\sum_{\begin{subarray}{l} q=1,2,\cdots, n;\\
 \gcd(q, n_0)=1
 \end{subarray}}\varphi(q)+\sum_{\begin{subarray}{l} q=1,2,\cdots, n;\\
 \gcd(q, n_0)=1
 \end{subarray}}\varepsilon(q),
\end{align*}
where the error terms $\varepsilon(q)$ satisfies $\displaystyle{\lim_{n\to+\infty}\frac{\varepsilon(n)}{\varphi(n)}=0}$.

Note that
\[
 \sum_{\begin{subarray}{l} q=1,2,\cdots, n;\\
 \gcd(q, n_0)=1  \end{subarray}}\varphi(q)=\sum_{q=1}^{n}\varphi(q)-\sum_{k=1}^{\lfloor \frac{n}{n_0}\rfloor}\varphi(k n_0),
\]
and that
\[
  \varphi(kn_0)=\left\{\begin{array}{ll}
                        \varphi(k)(n_0-1), \quad &\gcd(n_0,k)=1, \\
                        \varphi(k)n_0, \quad   &\gcd(n_0, k)=n_0.
                       \end{array}\right.
\]
Then,
\[
\sum_{q=1}^{n}\varphi(q)-\sum_{k=1}^{\lfloor \frac{n}{n_0}\rfloor}n_0\varphi(k)\le \sum_{\begin{subarray}{l} q=1,2,\cdots, n;\\
 \gcd(q, n_0)=1  \end{subarray}}\varphi(q)\le \sum_{q=1}^{n}\varphi(q)-\sum_{k=1}^{\lfloor \frac{n}{n_0}\rfloor}(n_0-1)\varphi(k).
\]
By Theorem \ref{thm: the order of euler's phi function}, $\Phi(n)=\sum\limits_{k=1}^{n}\varphi(k)=\frac{3}{\pi^2} n^2+O(n\log n)$, as $n\to +\infty$, we have
\begin{align*}
  & \sum_{q=1}^{n}\varphi(q)-\sum_{k=1}^{\lfloor \frac{n}{n_0}\rfloor}n_0\varphi(k)=(1-\frac{1}{n_0})\frac{3}{\pi^2} n^2+O(n\log n),\\
  & \sum_{q=1}^{n}\varphi(q)-\sum_{k=1}^{\lfloor \frac{n}{n_0}\rfloor}(n_0-1)\varphi(k)=(1-\frac{n_0-1}{n_0^2})\frac{3}{\pi^2}n^2+O(n\log n).
\end{align*}

Recall that $\displaystyle{\frac{\varepsilon (n)}{\varphi(n)}\to 0}$,  we have
$\displaystyle{\frac{\sum_{k=1}^{n}|\varepsilon(k)|}{n^2}\to 0}$, as $n\to+\infty$.
Therefore,
\begin{align*}
&\liminf_{n\to+\infty}\Psi(n; \rho^-, \rho^+)\frac{1}{n^2}\ge (\rho^+-\rho^-)\frac{3}{\pi^2} (1-\frac{1}{n_0})>0,\\
&\limsup_{n\to+\infty}\Psi(n; \rho^-, \rho^+)\frac{1}{n^2}\le (\rho^+-\rho^-)\frac{3}{\pi^2} (1-\frac{n_0-1}{n_0^2})<\infty.\qedhere
\end{align*}
\end{proof}

\subsection{Rotation number}\label{subsec: rotaion number}
In this section, we will introduce the rotation numbers for annulus homeomorphism. For more information,  refer to \cite{Franks88} and \cite{Lecalvez01}.

 We denote  by $\mathbb{A}$ the open  (resp.  the closed) annulus unless an explicit mention, i.e.,  $\mathbb{A}=\mathbb{R}/\mathbb{Z}\times (0,1)$ (resp. $\mathbb{A}=\mathbb{R}/\mathbb{Z}\times [0,1]$),  by $\pi$ the covering map of the  annulus
\begin{eqnarray*}
\pi\,:\, \mathbb{R}\times(0,1)\quad (\mathrm{resp.}\quad\mathbb{R}\times[0,1])&\rightarrow& \mathbb{A}\\
(x,y)&\mapsto&(x+\mathbb{Z},y),
\end{eqnarray*}
and by  $T$ the generator of the covering transformation group
\begin{eqnarray*}
T\,:\, \mathbb{R}\times(0,1)\quad (\mathrm{resp.}\quad\mathbb{R}\times[0,1])&\rightarrow&
\mathbb{R}\times(0,1)\quad (\mathrm{resp.}\quad\mathbb{R}\times[0,1]) \\
(x,y)&\mapsto&(x+1,y).
\end{eqnarray*}
Coordinates are denoted as $z\in\mathbb{A}$ and $\tilde{z}$ in the covering space. Homeomorphisms of $\mathbb{A}$ are denoted by $f$,
and their lifts to the covering space are denoted by $\widetilde{f}$.

Consider the homeomorphism $f$ of $\mathbb{A}$ that is isotopic to the identity.  We say that a positively recurrent point $z$ has a
\emph{rotation number} $\rho(\widetilde{f}, z)\in \mathbb{R}$ for a lift $\widetilde{f}$ of $f$ to the universal covering space of $\mathbb{A}$,
if for every subsequence $\{f^{n_k}(z)\}_{k\geq 0}$ of $\{f^n(z)\}_{n\geq 0}$ which converges to $z$, we have
\[\lim_{k\rightarrow+\infty}\frac{p_1\circ \widetilde{f}^{n_k}(\widetilde{z})-p_1(\widetilde{z})}{n_k}=\rho(\widetilde{f}, z), \]
where $\widetilde{z}\in \pi^{-1}(z)$ is a lift of $z$ and $p_1$ is the standard projection to the first coordinate.
The rotation number is stable by conjugacy  (see \cite{Lecalvez01}).
In particular,  if $z$ is a fixed or periodic point of $f$, the rotation number $\rho(\widetilde{f}, z)$ always exists and is rational.

A positively recurrent point of $f$ is also a positively recurrent point of $f^q$ for all $q\in \mathbb{N}$ (see \cite[Appendix]{Wang14}). By the definition of the rotation number, we easily get  the following elementary properties:
\begin{enumerate}\label{prop:ROT}
  \item[1.] $\rho(T^k\circ\widetilde{f}, z)=\rho(\widetilde{f}, z)+k$ for every $k\in \mathbb{Z}$;
  \item[2.] $\rho(\widetilde{f}^q, z)=q\rho(\widetilde{f}, z)$ for every $q\in \mathbb{N}$.
\end{enumerate}

We call a simple closed curve  in $\mathbb{A}$ an \emph{essential circle} if  it is not null-homotopic. We say that $f$ satisfies the \emph{intersection property} if  any essential circle in $\mathbb{A}$ meets its image by $f$. It is easy to see that a homeomorphism $f$ that preserves a finite area satisfies the intersection property.

We need the following theorem due to Franks \cite{Franks88} when $\mathbb{A}$ is the closed annulus and $f$ has no wandering point,  and  improved by Le Calevez \cite{Lecalvez05} (see also \cite{Wang14}) when $\mathbb{A}$ is the open annulus and $f$ satisfies the intersection property:

\begin{theorem}\label{thm: Franks and le calvez}
Let $f$ be a homeomorphism of $\mathbb{A}$ that is isotopic to the identity and satisfies the intersection condition,  and $\widetilde{f}$  one of its lifts to the universal covering space. Suppose that there exist two recurrent points $z_{1}$ and $z_{2}$ such that
$-\infty\leq \rho(\widetilde{f}, z_{1})<\rho(\widetilde{f}, z_{2})\leq +\infty$. Then for
any rational number $p/q\in (\rho(\widetilde{f}, z_{1}),\rho(\widetilde{f}, z_{2}))$
written in an irreducible way, there exists an  interior $q$ prime-periodic point with rotation number $p/q$.
\end{theorem}

\subsection{Transverse foliation and maximal isotopy}
Let $M$ be an oriented surface, and $\mathcal{F}$ an oriented topological foliation on $M$ whose leaves are oriented curves. We say that a path is \emph{positively transverse} to $\mathcal{F}$, if it meets the leaves of $\mathcal{F}$ locally from left to right.
Let $f$ be a homeomorphism on $M$, and $I=(f_t)_{t\in[0,1]}$  an identity isotopy of $f$, i.e., an isotopy joining the identity to $f$.
We say that an oriented foliation $\mathcal{F}$ (without singularity)  is a \emph{transverse foliation}  of $I$ if for every $z\in M$, there is a path  that is homotopic to the trajectory $t\rightarrow f_t(z)$ of $z$ along $I$ with the end points fixed and  is positively transverse to   $\mathcal{F}$.
If $f$ does not have any contractible fixed point associated to $I$, i.e., a fixed point of $f$ whose trajectory along $I$ is null homotopic in $M$, Le Calvez  proved the existence of the transverse foliation \cite[Theorem~1.3]{Lecalvez05}.

Now, we consider the case that $f$ has a contractible fixed point. We call  $K\subset \mathrm{Fix}(f)$ \emph{unlinked}, if there is an identity isotopy  $I=(f_t)_{t\in[0,1]}$ of $f$ such that $K\subset \mathrm{Fix}(I)=\cap_{t\in[0,1]}\mathrm{Fix}(f_t)$. We call an identity isotopy \emph{maximal}, if $\mathrm{Fix}(I)$ is maximal for including among all unlinked sets.  Moreover, if $I$ is a maximal isotopy, $f|_{M\setminus \mathrm{Fix}(I)}$ does not have any contractible fixed point. We consider a singular foliation $\mathcal{F}$ and call it  a \emph{transverse foliation} of $I$, if the set of singularities $\mathrm{Sing}(\mathcal{F})$ is equal to the fixed point set $\mathrm{Fix}(I)$ of $I$, and if $\mathcal{F}|_{M\setminus\mathrm{Sing}(\mathcal{F})}$ is transverse to $I|_{M\setminus \mathrm{Fix}(I)}$. Combine Le Calvez's result and the following theorem, we always get the existence of  a maximal isotopy $I$ and a transverse foliation $\mathcal{F}$.

\begin{theorem}\cite[Corollary 1.2]{BCLR}\label{thm: BCLR}
Let $f$ be a homeomorphism of $M$ and $K$ an unlinked set. Then,  there is a maximal isotopy $I$ such that $K\subset \mathrm{Fix}(I)$.
\end{theorem}

In particular, for area preserving homeomorphisms of the sphere, if the maximal isotopy has only finitely many fixed points, then the dynamics of the transverse foliation is quite simple: there is neither a closed leaf nor a leaf from and toward  the same singularity,  and hence every leaf joins one singularity to another singularity.

\subsection{Prime ends compactification and prime ends rotation number}\label{S: pre-prime-ends Compactification}
In this section, we will give  the definitions  and the result that we need for our paper.
More details can be found in \cite{KLN15}.

Let $U\subset \mathbb{S}^2$ be an open topological disk such that $\mathbb{S}^2\setminus U$ contains at least two points.
We can define the  \emph{prime ends compactification} of the open topological disk $U$, introduced by Carath\'{e}odory \cite{Car13},  by attaching a circle of prime ends $\backsimeq \mathbb{S}^1$ and topologizing $U\sqcup \mathbb{S}^1$ appropriately, making it homeomorphic to a closed disk.
The prime end compactification can be defined purely topologically (see \cite{Mat82} and \cite{KLN15} for more details), but has another significance if we put a complex structure on $\mathbb{S}^2$.
We can find a conformal map $\phi$ between $U$ and the open disk $\mathbb{D}=\{z\in \mathbb{C}\mid |z|<1\}$ and we put on $U\sqcup\mathbb{S}^1$ the topology (up to a homeomorphism of the resulting space, independent of $\phi$) induced from the natural topology of $\overline{\mathbb{D}}$ in $\mathbb{C}$ by the bijection
\[\overline{\phi}:U\sqcup\mathbb{S}^1\rightarrow \overline{\mathbb{D}},\]
which is equal to $\phi$ on $U$ and to the identity on $\mathbb{S}^1$,  where $\overline{\mathbb{D}}$ is the closure of $\mathbb{D}$.

Let $f$ be an orientation  preserving homeomorphism of $\mathbb{S}^2$ such that $f(U)=U$. Then, $f|_{U}$ can be extended to a homeomorphism of $U\sqcup \mathbb{S}^{1}$, which is still denoted by $f$. Moreover,
$f|_{\mathbb{S}^1}$ is an orientation preserving homeomorphism of the circle.  We define the \emph{prime ends rotation number} $\rho(f, U)\in \mathbb{R}/\mathbb{Z}$ of $f$ on the boundary of $U$  as the Poincar\'e's rotation number of $f|_{\mathbb{S}^1}$.

The following result is a direct corrollary of \cite[Theorem C]{KLN15}.
\begin{theorem}\label{thm: prime end rotation number}
  Let $f$ be an orientation and area preserving homeomorphism of the sphere $\mathbb{S}^2$, and $U\subset \mathbb{S}^2$  an $f$-invariant topological disk such that $\mathbb{S}^2\setminus U$ contains at least two points.  If $\rho(f,U)\ne 0$, then $f$ has at most one fixed point  in $\mathbb{S}^2\setminus U$.
\end{theorem}

\subsection[Cover $SO(3)$ by $SU(2)\cong \mathbb{S}^3$]{\texorpdfstring{Cover $SO(3)$ by $SU(2)\cong \mathbb{S}^3$}{Cover SO(3) by SU(2) = S3}}

In this section, we will give a classic covering map that covers $SO(3)$ by $SU(2)\cong \mathbb{S}^3$.  More details can be find in \cite[Section 1.6]{Sternberg}.

We can first identify
\[\mathbb{S}^3=\{(z_1, z_2): z_1, z_2\in\mathbb{C}, |z_1|^2+|z_2|^2=1\}\] with
$SU(2)$ by
\begin{align*}
  \mathbb{S}^3 & \to SU(2),\\
  (z_1, z_2) & \mapsto U=\begin{bmatrix}
            z_1 & z_2\\
            -\bar{z}_2 & \bar{z}_1
           \end{bmatrix}.
\end{align*}
Then, we will cover $SO(3)$ by $SU(2)$.

 Note that any $2\times 2$ traceless Hermitian matrix $M$ can be written as a linear combination of the Pauli matrices:
 \[\sigma_1=\begin{bmatrix}
              0 & 1\\
              1 & 0
            \end{bmatrix}, \quad
 \sigma_2=\begin{bmatrix}
           0 & -i \\
           i & 0
          \end{bmatrix}, \quad
 \sigma_3=\begin{bmatrix}
            1 & 0\\
            0 & -1
          \end{bmatrix}.\]
More explicitly,
\[M=\begin{bmatrix}
      c & a-i b\\
      a+i b & -c
    \end{bmatrix}=(\sigma_1, \sigma_2,\sigma_3)\begin{bmatrix}
                                                 a\\
                                                 b\\
                                                 c
                                               \end{bmatrix},\]
where $(a, b, c)^T\in \mathbb{R}^3$.
So, we can identify the linear space $V$  of  $2\times 2$ traceless Hermitian matrix  with $\mathbb{R}^3$ by choosing $\sigma_1$, $\sigma_2$, $\sigma_3$ as the basis of $V$, and can  define the norm on $V$ by
\[\|M\|_{\mathbb{R}^3}=\sqrt{-\det M}.\]

For any $U\in SU(2)$ and $M\in V$, $U^* M U$ is a Hermitian matrix with
\[\mathrm{tr}(U^* M U)= \mathrm{tr}(M)=0, \text{ and } \det(U^* M U)=\det M.\]
So, the map $M\mapsto U^* M U$ is an isometry of $V$, and hence corresponds to an orthogonal matrix in $O(3)$.
The group action
\begin{eqnarray*}
SU(2)\times V &\to & V,\\
(U, M) & \mapsto & U^* M U,
\end{eqnarray*} is continuous, and  induces a group homomorphism
\[R:  SU(2)  \to O(3).\]
Note that $R$ maps the unit matrix in $SU(2)$ to the unit matrix in $O(3)$. It is in fact a group homomorphism
\[R:  SU(2)  \to SO(3).\]
Moreover, it is indeed a two-covering map (see also \cite[Section 1.6]{Sternberg}).

Now, we have the  induced covering map
\[\pi : \mathbb{S}^3\to SO(3).\]
By computing $U^* \sigma_i U$, $i=1,2,3$,  we have
\[\pi(z_1,z_2)= \begin{bmatrix}
           \mathrm{Re}(z_1^2-\bar{z}_2^2) & - \mathrm{Im}(z_1^2+\bar{z}_2^2) & 2  \mathrm{Re}(z_1\bar{z}_2)\\
            \mathrm{Im}(z_1^2-\bar{z}_2^2) &  \mathrm{Re}(z_1^2+\bar{z}_2^2)  & 2 \mathrm{Im} (z_1 \bar{z_2})\\
           -2 \mathrm{Re}(z_1 z_2) & 2  \mathrm{Im}(z_1 z_2) & |z_1|^2- |z_2|^2
           \end{bmatrix} . \]

In conclusion, we have the following result:
\begin{lemma}\label{lemma: cover SO(3)}
The unit  $3$-dimensional sphere
\[\mathbb{S}^3=\{(z_1, z_2): z_1, z_2\in\mathbb{C}, |z_1|^2+|z_2|^2=1\}\]
 covers  $SO(3)$ twice by  the covering map
 \begin{eqnarray*}
  \pi: &\mathbb{S}^3 & \to SO(3),\\
  & (z_1, z_2) &\mapsto
           \begin{bmatrix}
           \mathrm{Re}(z_1^2-\bar{z}_2^2) & - \mathrm{Im}(z_1^2+\bar{z}_2^2) & 2  \mathrm{Re}(z_1\bar{z}_2)\\
            \mathrm{Im}(z_1^2-\bar{z}_2^2) &  \mathrm{Re}(z_1^2+\bar{z}_2^2)  & 2 \mathrm{Im} (z_1 \bar{z_2})\\
           -2 \mathrm{Re}(z_1 z_2) & 2  \mathrm{Im}(z_1 z_2) & |z_1|^2- |z_2|^2
           \end{bmatrix}.
 \end{eqnarray*}
\end{lemma}

\section{Proof of the main theorem}

In this section, we consider the homeomorphism $f$ of the closed or open annulus that is isotopic to the identity. We will prove the main theorem (Theorem \ref{thm: growth of periodic orbits}) of the paper and the following Theorem \ref{thm: growth of n periodic}.

\begin{theorem}\label{thm: growth of n periodic}
  Let $f$ be a homeomorphism of the closed or open annulus, that is isotopic to the identity and preserves a finite area. Suppose that $f$ has a fixed or  periodic  point, and is of infinite order, i.e., there is not any positive integer $m$ such that $f^m=\mathrm{Id}$. Then, one of the following three cases happens:
   \begin{itemize}
     \item[i)]  $f$ has a fixed point,
     \item[ii)] $f$ does not have any fixed point, and have  periodic orbits with prime-periods $n_1, n_2, \cdots, n_s$ such that $\gcd(n_1, n_2,\cdots, n_s)=1$,
         \item [iii)] the periods of  all the periodic orbits have a greatest common divisor $k\ge 2$.
   \end{itemize}
    In case $i)$ and $ii)$, we have
  \[\liminf_{n\to +\infty} N_{=n}\frac{\log\log n}{n}>0.\]
In case $iii)$, we have,
 \[\liminf_{n\to +\infty} N_{=kn}\frac{\log\log n}{n}>0.\]
\end{theorem}

 We first give the  idea of the proof of Theorem \ref{thm: growth of periodic orbits} and Theorem \ref{thm: growth of n periodic}. If $f$ is of finite order, it is conjugate to a rational rotation of the annulus (see \cite{Bro19, Eil34,Ker36}). The results are easy to prove.
 If $f$ is of infinite order and has a fixed or periodic point, one of the three cases in Theorem \ref{thm: growth of n periodic} happens.
  In case i), we will prove the existence of  different rotation numbers for the homeomorphism  or the induced homeomorphism on a modified annulus, and get the results by the following Lemma \ref{lemma: different rotation implies growth}.
  In case ii), we can also prove the existence of  different rotation numbers in Lemma \ref{lemma: case 2 implies diffrent rotation number}, and get the results by the Lemma \ref{lemma: different rotation implies growth}.
  In case iii),  we consider $f^k$, which satisfies the condition of case i) or ii),   and compare the prime-periods of a periodic point for  $f$ and for $f^k$ respectively.

Now, we begin with the lemmas, and then prove Theorem~\ref{thm: growth of n periodic} and Theorem~\ref{thm: growth of periodic orbits}.

\begin{lemma}\label{lemma: different rotation implies growth}
Let $f$ be a homeomorphism of $\mathbb{A}$ that is isotopic to the identity and satisfies the intersection condition,  and $\widetilde{f}$  one of its lifts to the universal covering space. Suppose that there exist two positive recurrent points $z_{1}$ and $z_{2}$ such that
$-\infty\leq \rho(\widetilde{f}, z_{1})=\rho_1<\rho(\widetilde{f}, z_{2})=\rho_2\leq +\infty$. Then,
  \[\liminf_{n\to+\infty} N_{=n}\frac{\log\log n}{n}>0, \quad \liminf_{n\to+\infty}N_{\le n}\frac{1}{n^2}>0,\quad \liminf_{n\to+\infty}N_{\le n, n_0\nmid} \frac{1}{n^2} >0,\]
where $n_0$ is a given prime number.
\end{lemma}

\begin{remark}
The homeomorphism, that preserves a finite area,  satisfies the intersection condition.
\end{remark}

\begin{proof}[Proof of Lemma \ref{lemma: different rotation implies growth}]
By Theorem \ref{thm: Franks and le calvez}, $f$ has an interior $n$ prime-periodic orbit with the rotation number $m/n$ for every irreducible $\frac{m}{n}\in( \rho_1, \rho_2 )$.
   So,
   \[N_{=n}\ge \varphi(n; \rho_1, \rho_2), \quad  N_{\le n}\ge \Phi(n; \rho_1, \rho_2), \quad N_{\le n, n_0\nmid}\ge \Psi(n; \rho_1, \rho_2).\]
   By Theorem \ref{thm: the order of euler's phi function} and Lemma \ref{Lemma: order of phi(n; rho-, rho+)},
   \[\liminf_{n\to +\infty}  N_{=n}\frac{\log\log n}{n}\ge  \liminf_{n\to +\infty} (\rho_2-\rho_1)\varphi(n)\frac{\log\log n}{n}=(\rho_2-\rho_1)e^{-\gamma}>0,\]
   where $\gamma$ is the Euler constant.

By Lemma \ref{lemma: order of Phi(n; rho-, rho+)},
   \[\liminf_{n\to+\infty}N_{\le n} \frac{1}{n^2} \ge \frac{3(\rho_2-\rho_1)}{\pi^2}>0.  \]

   By Lemma \ref{lemma: order of Psi(n; rho-, rho+)},
      \[\liminf_{n\to+\infty}N_{\le n, n_0\nmid} \frac{1}{n^2} >0. \qedhere \]
\end{proof}

\begin{lemma}\label{lemma: get nonzero rotation number}
Let $f$ be a homeomorphism of the  open annulus $\mathbb{A}$ that  preserves a finite area,  $I$ a maximal isotopy of $f$, $\widetilde{f}$ the lift of $f$ to the universal covering space of $\mathbb{A}$ associated with $I$ (i.e., $(\widetilde{f}_t)_{t\in[0,1]}$ is the lift of $I=(f_t)_{t\in[0,1]}$, $\widetilde{f}_0=Id$, $\widetilde{f}_1=\widetilde{f}$), and $\mathcal{F}$ a transverse foliation of $I$.  If there is a leaf of $\mathcal{F}$ from one end of the annulus to the other end of the annulus, then there exists a positive recurrent point with $\rho(\widetilde{f}, z)\ne 0$.
\end{lemma}

\begin{proof}
  We choose a leaf $\Gamma$ from one end of the annulus to the other end of the annulus and one of its lifts $\widetilde{\Gamma}$.   Note that the oriented curve $\widetilde{\Gamma}$ separates the covering space into two parts, and  $\widetilde{f}$ maps the  part on the right of $\widetilde{\Gamma}$  to a proper subset of itself.
 We can choose a small disk $U$ near $\Gamma$ and $\widetilde{U}$ the lift of $U$ near $\widetilde{\Gamma}$, such that $\widetilde{U}$ is on the left of $\widetilde{\Gamma}$ and $\widetilde{f}(\widetilde{U})$ is on the right of $\widetilde{\Gamma}$.
  We can also assume that  the diameter of  $p_1(\widetilde{U})$ is smaller than $\displaystyle{\frac{1}{2}}$  by choosing $U$ small enough, where $p_1$ is the projection to the first factor.
 We suppose that
 \[p_1(\widetilde{z'})-p_1(\widetilde{z})>\frac{1}{2} \text{ for all } \widetilde{z}\in \widetilde{U}  \text{ and } \widetilde{z'}\in \widetilde{U}',\]
  where $\widetilde{U}'$ is any lift of $U$ on the right of $\widetilde{\Gamma}$ (the proof in the case $p_1(\widetilde{z'})-p_1(\widetilde{z})<-\frac{1}{2}$ is similar).
  By Poincar\'e's recurrent theorem, almost all points (with respect to the measure induced by the given finite area) in $U$  are positive recurrent by $f$.
   We denote by $\mathrm{Rec}^+(f)$ the set of positive recurrent points of $f$,  by $\Phi$  the first return map for $f$ on  $U\cap \mathrm{Rec}^+(f)$,  and by $\tau$ the first return time.
 Let
 \[m(z)=p_1\circ\widetilde{f}^{\tau (z)}(\widetilde{z})-p_1(\widetilde{z}), \quad  z\in U\cap \mathrm{Rec}^+(f).\]
 Then,
 \[m(z)> \frac{1}{2},\]
  because the positive orbit of $\widetilde{z}$ by $\widetilde{ f}$ is on the right of $\widetilde{\Gamma}$.
  Let
  \[\tau_n(z)=\sum\limits_{i=0}^{n-1}\tau(\Phi^i(z)),\quad \text{ and }
    m_n(z)=\sum\limits_{i=0}^{n-1}m(\Phi^i(z)).\]
  Then,
  \[\frac{m_n(z)}{\tau_n(z)}\to \rho(\widetilde{f},z)\quad \text { as } n\to+\infty, \]
   if the rotation number exists.
 Note that
  \[\frac{m_n(z)}{\tau_n(z)}=\frac{m_n(z)}{n}\frac{n}{\tau_n(z)},\quad  \frac{m_n(z)}{n}>\frac{1}{2}.\]
   By  Kac's Lemma (see \cite{Kac47} and  \cite{Wright}), $\tau\in L^1(U,\mathbb{R})$. Then, $\frac{\tau_n}{n}$ is convergent a.e. by Birkhoff's ergordic theorem.
    Therefore,
    \[\rho(\widetilde{f},z)>0  \quad \text{ for  a.e. } z\in U. \qedhere\]
\end{proof}

\begin{lemma}\label{lemma: case 2 implies diffrent rotation number}
Let $f$ be a homeomorphism of the closed or open annulus, that is isotopic to the identity and preserves a finite area. If $f$ does not have any fixed point, and have  periodic orbits $O_1$, $O_2$, $\cdots$ $O_s$ with prime-periods $n_1, n_2, \cdots, n_s$ such that $\gcd(n_1, n_2,\cdots, n_s)=1$, then there exist two positive recurrent points with different rotation numbers.
\end{lemma}

\begin{proof}
If there exist $z_i\in O_i$, $z_j\in O_j$ such that $\rho(\widetilde{f},z_i)\ne \rho(\widetilde{f},z_j)$, we finish the proof. Now, we suppose that $\rho(\widetilde{f}, z)=\frac{p}{q}$ for all $z\in O_i$, $i=1,2,\cdots,s$, where $\gcd(p,q)=1$.   Since the prime-period of  $O_i$ is  $n_i$, we have $q| n_i$,  for $i=1,2,\cdots, s$. So, $q=1$, and hence $\rho(\widetilde{f}, z)$ is an integer. By composing a covering transformation of the universal covering space to the lift of $f$ if necessary, we can suppose that $\rho(\widetilde{f}, z)=0$. Now, we will prove the existence of a non-zero rotation number.

We consider the homeomorphism of the open annulus. When $\mathbb{A}$ is closed, we consider the restriction of $f$ to the interior of $\mathbb{A}$. Let $I$  be an identity isotopy  of $f$  such that the lift of $f$ associated with $I$ to the universal covering space is equal to $\widetilde{f}$.
The homeomorphism $f$ does not have any fixed point. So, $I$  does not have any fixed point and is a maximal isotopy. We consider the transverse foliation $\mathcal{F}$ of $I$.  All the leaves of $\mathcal{F}$ are  topological lines from one end of the annulus to the other end.
 We get the result by Lemma~\ref{lemma: get nonzero rotation number}.
\end{proof}

\begin{lemma}\label{lemma: period of fk}
  If the periods of  all the periodic orbits of $f$ has a greatest common divisor $k\ge 2$, then an $n$ prime-periodic point of $f^k$ is a $kn$ prime-periodic point of $f$.
\end{lemma}

\begin{proof}
If $z$ is an $n$ prime-periodic point of $f^k$, then $f^{kn}(z)=z$ and $f^{km}(z)\ne z$ for every positive integer   $m<n$.
So, $z$ is a periodic point of $f$, and the prime-period  $n'| kn$. We will prove $n'=kn$.

We first prove $\gcd(\frac{kn}{n'}, n)=1$. In fact, if $\frac{kn}{n'}$ and $n$ have a common divisor $t$, then $k \frac{n}{t}=\frac{kn/n'}{t} n'$. So, the prime-period  of $z$ for $f^k$ is a divisor of $\frac{n}{t}$. Hence, $t=1$.

Now, we suppose $n'= sn$, where $s$ is a positive integer.
For all prime  $p|k$,
\begin{itemize}
  \item if $p|n$, then $p\nmid \frac{kn}{n'}=\frac{k}{s}$.
  \item if $p\nmid n$, then $p| s$ (because $k| n'$). Moreover, if $p^r|k$,  then $p^r|s$. So, $p\nmid \frac{k}{s}$.
\end{itemize}
Therefore, $\frac{k}{s}=1$, and  $n'=n$.
\end{proof}

\begin{proof}[Proof of Theorem \ref{thm: growth of n periodic}]
  Suppose that $f$ is a homeomorphism of the closed or open annulus, that is isotopic to the identity, preserves a finite area, and is of infinite order. We will prove the theorem in the three cases one by one.
\begin{itemize}
\item[i)] We suppose that $f$ has a fixed point.

By Lemma \ref{lemma: different rotation implies growth}, if $f$ has positive recurrent points with different rotation numbers, we finish the proof.  Now, we suppose that the rotation number $\rho(\widetilde{f}, z)$ is a constant, where $\widetilde{f}$ is a lift of $f$ and $z$ is any positive recurrent point of $f$. Moreover, for a fixed point $z$ of $f$, the rotation number $\rho(\widetilde{f}, z)$ is an integer.  By composing a covering transformation  of the universal covering space to the lift of $f$ if necessary, we can suppose that $\rho(\widetilde{f}, z)=0$ for every positive recurrent point $z$.

 If $\mathbb{A}$ is a closed annulus, we shrink each boundary to a point and get a sphere; while $\mathbb{A}$ is an open annulus, we consider the end  compactifications and also get a sphere. Denote the sphere by $\mathbb{S}^2$, and the two points not in the interior of  $\mathbb{A}$ by $N$ and $S$. Moreover, $f$ induces an area preserving homeomorphism of $\mathbb{S}^2$ that fixes both $N$ and $S$. We still denote by $f$ the induced homeomorphism of $\mathbb{S}^2$

  By Theorem \ref{thm: BCLR}, there is a maximal isotopy $I=(f_t)_{t\in [0,1]}$ that fixes both $N$ and $S$, and the lift  of $f|_{\mathbb{S}^2\setminus \{N,S\}}$ associated with $I$ is $\widetilde{f}$. If $I$ have only two fixed points $N$ and $S$, we consider the transverse foliation $\mathcal{F}$ of $I$. The leaves of $\mathcal{F}$ join the two ends of the annulus.
   We get the existence of non-zero rotation numbers by Lemma \ref{lemma: get nonzero rotation number}, which contradicts with our assumption that $\rho(\widetilde{f}, z)=0$ for all positive recurrent points $z$. Therefore, $I$ has at least $3$ fixed points.

 Now, we choose a leaf $\Gamma$ of the transverse foliation $\mathcal{F}$ of $I$, and will find an invariant topological annulus of $f$ such that $\Gamma$ joins the two ends of the annulus.
  We first choose the connected component $U$ of $\mathbb{S}^2\setminus\mathrm{Fix}(I)$ containing $\Gamma$.  Note that the $\alpha$-limit set $\alpha(\Gamma)$  of $\Gamma$ and the $\omega$-limit set $\omega(\Gamma)$ of $\Gamma$  are connected compact subsets of the set of the singular points $\mathrm{Sing}(\mathcal{F})=\mathrm{Fix}(I)\subset \mathbb{S}^2\setminus U$ of $\mathcal{F}$ respectively.  We  fill $U$ as follows:
 Let $K_1$ and $K_2$ be the connected components of $\mathbb{S}^2\setminus U$ that contain  $\alpha(\Gamma)$  and  $\omega(\Gamma)$  respectively. Then,
 $\mathbb{A}'=\mathbb{S}^2\setminus(K_1\cup K_2)$ is an open set containing $U$, and is homeomorphic to a disk (if $K_1=K_2$) or  an annulus (if $K_1\ne K_2$).
 We consider the induced  homeomorphism, the induced maximal isotopy, and the induced transverse foliation on $\mathbb{A}'$, and still denote them by $f$, $I$, $\mathcal{F}$ respectively.

 \begin{lemma}
$K_1\ne K_2$.
 \end{lemma}
 \begin{proof}
 We will prove the lemma by contradiction. Suppose that $K_1=K_2$. Then, $\mathbb{A}'=\mathbb{S}^2\setminus(K_1\cup K_2)$ is a topological disk.
 By considering the one point compactification of $\mathbb{A}'$, we get a sphere $\mathbb{A}'\cup \{\star\}$, the induced area-preserving homeomorphism $f$, the induced maximal isotopy $I$, and the induced transverse foliation $\mathcal{F}$ on the sphere $\mathbb{A}'\cup \{\star\}$.
 But the leaf $\Gamma$ is from $\star$ to $\star$, which contradicts with the area preserving condition.
 \end{proof}

So,  $\mathbb{A}'=\mathbb{S}^2\setminus(K_1\cup K_2)$ is a topological open annulus and $\Gamma$ is a leaf joining two ends of $\mathbb{A}'$.
Let $\widehat{f}$ be the lift of $f$ to the universal covering space of $\mathbb{A}'$ associated with $I$.
By Lemma \ref{lemma: get nonzero rotation number}, there is a positive recurrent point $z$ with $\rho(\widehat{f},z)\ne 0$.   If $I$ has a fixed point $z'\in \mathbb{A}'$, then $\rho(\widehat{f},z')=0$. We get different rotation numbers, and can finish the proof by Lemma \ref{lemma: different rotation implies growth}.

Now, we suppose that $I$ does not have any fixed point in $\mathbb{A}'$. Since the original maximal isotopy on the sphere $\mathbb{S}^2$ has at least $3$ fixed points, either $K_1$ or $K_2$ contains at least two fixed points of $I$.  We suppose that $K_1$ contains at least two fixed points of $I$, which are also fixed points of $f$.
We consider the prime ends compactification of $\mathbb{S}^2\setminus K_1$,  and extend $f$ to the closed disk $(\mathbb{S}^2\setminus K_1)\sqcup \mathbb{S}^1$. By Theorem~\ref{thm: prime end rotation number}, the prime ends rotation number of $f$ on the boundary is  equal to $0\in\mathbb{R}/\mathbb{Z}$.  Indeed, we get an annulus $\mathbb{A}'\sqcup \mathbb{S}^1$ with one  boundary and the induced homeomorphism  (still denoted  by $f$), such that $\rho(\widehat{f}, z)$ is equal to an integer, for all $z\in \mathbb{S}^1$.
Recall that the boundary of $\mathbb{A}'$ ( $\subset \mathbb{S}^2$) is a subset of $\mathrm{Fix}(I)$, where $I=(f_t)_{t\in [0,1]}$ and $f_0=\mathrm{Id}$.
We can also extend $f_t|_{\mathbb{S}^2\setminus K_1}$ to $\mathbb{S}^1$, and  consider the prime ends rotation number of $f_t$ on the boundary of $\mathbb{S}^2\setminus K_1$. We get the induced $I=(f_t)_{t\in[0,1]}$ on $\mathbb{A}'\sqcup \mathbb{S}^1$. Moreover, the rotation number $\rho(\widehat{f}_t, z)$ is equal to an integer, for all $z\in \mathbb{S}^1$. Because $\rho(\widehat{f}_t, z)$ is continuous for $t$, and $\widehat{f}_0=\mathrm{Id}$,   we have  $\rho(\widehat{f}_t, z)=0$ for $z\in\mathbb{S}^1$ and $t\in [0,1]$. In particular, $\rho(\widehat{f}, z)=0$ for $z\in\mathbb{S}^1$.
By Lemma \ref{lemma: get nonzero rotation number}, there is a positive recurrent point $z'$ with $\rho(\widehat{f},z')\ne 0$. We get different rotation numbers, and can finish the proof by Lemma \ref{lemma: different rotation implies growth}.

 \item[ii)]  Suppose that $f$ does not have any fixed point, and have  periodic orbits with prime-periods $n_1$, $n_2$, $\cdots$, $n_s$ such that $\gcd(n_1, n_2,\cdots, n_s)=1$.
 In this case, the theorem is a corollary of Lemma~\ref{lemma: different rotation implies growth} and Lemma~\ref{lemma: case 2 implies diffrent rotation number}.

 \item[iii)]Suppose that the periods of  all the periodic orbits of $f$ have a greatest common divisor $k\ge 2$.
  Let $g=f^k$. Then, $g$ is  a homeomorphism of the closed or open annulus, that is isotopic to the identity, preserves a finite area, and is of infinite order.

If $g$ has a fix point, the result is a corollary of   case i)  and Lemma~\ref{lemma: period of fk}.

If $g$ does not have any fixed point, then it has periodic orbits with different prime-periods and the greatest common divisor  of the periods of  all the periodic orbits of $g$ is equal to $1$. So, $g$ satisfies the condition of  case ii). We get the result as a corollary of case ii) and Lemma~\ref{lemma: period of fk}.\qedhere
 \end{itemize}
\end{proof}

\begin{proof}[Proof of Theorem \ref{thm: growth of periodic orbits}]
If  $f$ is of finite order, there exists a positive integer $m$ such that  $f^m=\mathrm{Id}$. Moreover, $f$ is conjugate to a rational rotation of the annulus (see \cite{Bro19, Eil34,Ker36}).
By choosing the minimal $m$, we can suppose that every point is  an $m$ prime-period point.
  Then, $N_{\le n}=+\infty$ for all $n\ge m$.
  If $f$ has a fixed or periodic point with prime-periodic $k$ such that $(k, n_0)=1$, then $k=m$, and hence $(m, n_0)=1$,
   $N_{\le n, n_0\nmid}=+\infty$ for all $n\ge m$.  The third condition, that $f$ has a $k$ prime-periodic point and $f^k\ne \mathrm{Id}$ cannot happen. We have nothing to prove.

Now, we suppose that $f$ is of infinite order and has a fixed or periodic point. Then, one of the three cases in Theorem \ref{thm: growth of n periodic} happens. We will prove the theorem in the three cases one by one.

\begin{itemize}
 \item[i)]Suppose that $f$ has a fixed point  (the case i) of Theorem \ref{thm: growth of n periodic}).
     We have in fact proved the existence of two different rotation numbers for the homeomorphism in the annulus or in a modified annulus in the proof of the case i) of Theorem \ref{thm: growth of n periodic}. Then, we get the results by Lemma~\ref{lemma: different rotation implies growth}.

\item[ii)] Suppose that $f$ does not have any fixed point, and has  periodic orbits with prime-periods $n_1, n_2, \cdots, n_s$ such that $\gcd(n_1, n_2,\cdots, n_s)=1$ (the case ii) of Theorem \ref{thm: growth of n periodic}).
We  proved the existence of two different rotation numbers in Lemma \ref{lemma: case 2 implies diffrent rotation number}. Then, we get the results by Lemma~\ref{lemma: different rotation implies growth}.

\item[iii)]
Suppose that the periods of  all the periodic orbits have a greatest common divisor $k'\ge 2$ (the case iii) of Theorem \ref{thm: growth of n periodic}).
 We have in fact proved $N_{=k'n}\ge \varphi(n; \rho^-,\rho^+)$ for some $\rho^-<\rho^+$ in the proof of the case iii) of Theorem \ref{thm: growth of n periodic}. Therefore,
  \[N_{\le n}=\sum_{j=1}^{\lfloor\frac{n}{k'}\rfloor} N_{=k'j}\ge \Phi(\lfloor\frac{n}{k'}\rfloor; \rho^-,\rho^+), \]
where $\lfloor\frac{n}{k'}\rfloor$ is the maximal integer not exceeding $\frac{n}{k'}$.
We get the first result by Lemma \ref{lemma: order of Phi(n; rho-, rho+)}.

 Moreover, if $f$ has a periodic orbit with the prime-period $k$ with $(k, n_0)=1$, then  $(k', n_0)=1$. So, $(k'j, n_0)=1 \Leftrightarrow (j,n_0)=1$.
\[N_{\le n, n_0\nmid }=\sum_{\begin{array}{l} j=1,2,\cdots,\lfloor\frac{n}{k'}\rfloor ;\\
 \gcd (j, n_0)=1\end{array}} N_{=k'j}\ge \Psi(\lfloor \frac{n}{k'}\rfloor ; \rho^-, \rho^+).\]
We get the second result by Lemma \ref{lemma: order of Psi(n; rho-, rho+)}.

The third result is obviously, because $k'| k$. \qedhere
\end{itemize}
\end{proof}

\section[Non-contractible closed geodesics on Riemannian or Finsler $\mathbb{R}P^2$]{\texorpdfstring{Non-contractible closed geodesics on Riemannian or Finsler $\mathbb{R}P^2$}{Non-contractible closed geodesics on Riemannian or Finsler RP2}}
In this section, we will use Theorem \ref{thm: growth of periodic orbits} to
study the multiplicity and also the growth rate of the number of non-contractible closed geodesics on Riemannian or Finsler $\mathbb{R}P^2$.
Firstly, we review some background on this topic.

A closed curve on a Finsler manifold is a closed geodesic if it is locally the shortest path connecting any two nearby points on this curve. As usual, on any Finsler manifold $(M, F)$, a closed geodesic $c: \mathbb{S}^1=\mathbb{R}/ \mathbb{Z} \to M$ is \emph{prime} if it is not a multiple covering (i.e., iteration) of any other closed geodesics.
Here, the $m$-th iteration $c^m$ of $c$ is defined by $c^m(t)=c(mt)$.  The inverse curve $c^{-1}$ of $c$ is defined by $c^{-1}(t)=c(1-t)$ for $t\in \mathbb{R}$.  Note that unlike Riemannian manifold, the inverse curve $c^{-1}$ of a closed geodesic $c$ on an irreversible Finsler manifold need not be a geodesic.
We call two prime closed geodesics $c$ and $d$ {\it distinct} if there is no $\theta\in (0,1)$ such that $c(t)=d(t+\theta)$ for all $t\in\mathbb{R}$. For a closed geodesic $c$ on $(M,\,F)$, we denote by $P_c$ the linearized Poincar\'{e} map of $c$. Recall that a Finsler metric $F$ is \emph{bumpy} if all the closed geodesics on $(M, \,F)$ are non-degenerate, i.e., $1\notin \sigma(P_c)$ for any closed geodesic $c$. Following Rademacher in \cite{Rad2004}, we define  the reversibility $\lambda=\lambda(M,\,F)$ of a compact Finsler manifold $(M,\,F)$ to be
\[\lambda:=\max\{F(-X)\,|\,X\in TM, \,F(X)=1\}\ge 1.\]

There is a famous conjecture in Riemannian geometry which claims that there exist infinitely many closed geodesics on any compact Riemannian manifold. This conjecture has been proved except for CROSS's (compact rank one symmetric spaces); cf. \cite{GM1969JDG} and \cite{VS1976}. The results of Franks \cite{Franks92} in 1992 and Bangert \cite{Bangert1993} in 1993 imply that this conjecture is true for any Riemannian 2-sphere (cf. also  Hingston \cite{Hingston1993}).

But once one moves to the Finsler case, the conjecture becomes false. In 1973, Katok \cite{Katok1973} endowed some irreversible Finsler metrics to the compact rank one symmetric spaces
\[\mathbb{S}^{n},\ \mathbb{R}P^{n},\ \mathbb{C}P^{n},\ \mathbb{H}P^{n}\ \text{and}\ {\rm CaP}^{2},  \label{mflds}\]
such that each of the spaces possesses only finitely many distinct prime closed geodesics (cf. also Ziller \cite{Ziller1982}).
In particular, the number of closed geodesics on $\mathbb{S}^{n}$ and $\mathbb{R}P^{n}$ with Katok's metrics is equal to $2[\frac{n+1}{2}]$.

In 2004,  Bangert and Long \cite{BL2010} (published in 2010) proved the existence of at least two distinct closed geodesics on every Finsler $\mathbb{S}^2$. Since then, there are many interesting results about the multiplicity of closed geodesics on Finsler spheres and compact simply-connected Finsler manifolds; cf. \cite{DLW} and the references therein.

As for the multiplicity of closed geodesics on non-simply connected manifolds whose free loop space possesses bounded Betti number sequence, Ballman et al. \cite{BTZ1} proved in 1981 that every Riemannian manifold, with the fundamental group being a nontrivial finitely cyclic group and possessing a generic metric, has infinitely many distinct closed geodesics. In 1984, Bangert and Hingston \cite{BH} proved that any Riemannian manifold, with the fundamental group being an infinite cyclic group, has infinitely many distinct closed geodesics. Since then, there seem to be very few works on the multiplicity of closed geodesics on non-simply connected manifolds. The main reason is that the topological structures of the free loop spaces on these manifolds are not well known, so that the classical Morse theory is hardly applicable.

Recently, there are some studies on the multiplicity of non-contractible closed geodesics on Finsler $\mathbb{R}P^n$, in particular, Xiao and Long \cite{XL2015} in 2015 investigated the topological structure of the non-contractible loop space and established the resonance identity for the non-contractible closed geodesics on $\mathbb{R} P^{2n+1}$ by use of $\mathbb{Z}_2$ coefficient homology.
 As an application, Duan, Long and Xiao \cite{DLX2015} proved the existence of at least two distinct non-contractible closed geodesics on $\mathbb{R} P^{3}$ endowed with a bumpy and irreversible Finsler metric.  Subsequently in \cite{Tai2016}, Taimanov used a quite different method from \cite{XL2015} to compute the rational equivariant cohomology of the non-contractible loop spaces in compact space forms $\mathbb{S}^n/ \Gamma$, and proved the existence of at least two distinct non-contractible closed geodesics on $\mathbb{R}P^2$ endowed with a bumpy irreversible Finsler metric, where $\Gamma$ is a finite group which acts freely and isometrically on the $\mathbb{S}^n$.
 Note that the only non-trivial group which acts freely on $\mathbb{S}^{2n}$ is $\mathbb{Z}_2$ and that $\mathbb{S}^{2n}/ \mathbb{Z}_2=\mathbb{R}P^{2n}$ (cf. P.5 of \cite{Tai2016}).
 Motivated by Taimanov \cite{Tai2016}, Liu \cite{Liu} proved that there exist at least $2[\frac{n+1}{2}]$ prime non-contractible closed geodesics on every bumpy and irreversible Finsler $(\mathbb{R}P^n,F)$ with reversibility $\lambda$ and flag curvature $K$ satisfying $\frac{64}{25}\left(\frac{\lambda}{1+\lambda}\right)^2<K\le 1$.
 In \cite{LX}, Liu and Xiao established the resonance identity for the non-contractible closed geodesics on $\mathbb{R}P^n$, and together with \cite{DLX2015} and \cite{Tai2016} proved the existence of at least two distinct non-contractible closed geodesics on every bumpy $\mathbb{R}P^n$ with $n\geq2$.
Recently, Rademacher and Taimanov~\cite{RT22} studied the multiplicity of non-contractible closed geodesics on  compact Riemannian or Finsler manifold with infinite fundamental group.

On the other hand, Liu, Wang and Yan \cite{LWY} proved some refinements of Franks' theorem and gave some applications
on the two or infinity results about the closed Reeb orbits with symmetries in contact geometry:

\begin{lemma} \cite[Theorem 11]{LWY}\label{lm: 1}
Let $(\mathbb{S}^3, \lambda, \xi_0)$ be a dynamically convex tight three-sphere satisfying $g_{p,1}^*\lambda = \lambda$ for some $p \geq1$, where $\xi_0$ is the standard contact structure and $g_{p,1} : \mathbb{S}^{3} \rightarrow \mathbb{S}^{3}$ is defined by
\[g_{p,1}(z_1=x_1+iy_1, z_2=x_2+iy_2) = (e^{2\pi i/p}z_1, e^{2\pi i/p}z_2),\]
via the identification $\mathbb{C}^2=\mathbb{R}^4$. Then there exist two or infinitely many $g_{p,1}$-symmetric periodic orbits on $\mathbb{S}^{3}$.
\end{lemma}

This result has a counterpart on the closed geodesic problem, i.e., there exist either two or infinitely many closed geodesics on Finsler $\mathbb{S}^2$ with $K \geq 1$ for which every geodesic loop is longer than $\pi$ (cf. \cite{HaP1}). For every Riemannian $\mathbb{S}^2$, as mentioned above, Franks \cite{Franks92} and Bangert \cite{Bangert1993} proved that there exist infinitely many closed geodesics, and Hingston \cite{Hingston1993} further proved that the number of closed geodesics of length $\leq l$ grows at least like the prime numbers.

Motivated by the above results, we can use Theorem \ref{thm: growth of periodic orbits} to prove a growth rate about the non-contractible closed geodesics on Riemannian $\mathbb{R}P^2$  under some natural assumptions, i.e., Theorem \ref{thm: application 1}.

We first give the idea of the proof of Theorem \ref{thm: application 1}. We consider the universal $2$-cover $(\mathbb{S}^2, g)$ of $(\mathbb{R}P^2, g)$, and the geodesic flow on $T\mathbb{S}^2$.  We restrict the geodesic flow on the unit tangent bundle $S_g\mathbb{S}^2$, and will get a global surface of section $\Sigma$, which is homeomorphic to the closed annulus.  Then, by composing the Poincar\'{e} half-return map with an involution $h_*$, we get an area-preserving homeomorphism $f$ of $\Sigma$ that is isotopic to the identity, and can extend it continuously to the boundary such that  $f|_{\partial \Sigma}=\mathrm{Id}$. Moreover, the odd periodic orbits of $f$ corresponds to non-contractible closed geodesics of $(\mathbb{R}P^2, g)$. Then, we get the result by  Theorem \ref{thm: growth of periodic orbits}.

\begin{proof}[Proof of Theorem \ref{thm: application 1}]
We consider the universal $2$-cover of $(\mathbb{R}P^2, g)$ and the induced Riemannian metric,  i.e. $(\mathbb{S}^2, g)$.
Let $h$ be the generator of $\pi_1(\mathbb{R}P^2)=\mathbb{Z}_2$.
 Then $h$  induces an involutive orientation-reversing isometry on $(\mathbb{S}^2,g)$, that is still denoted by $h$ and acts on $(\mathbb{S}^2,g)$ freely.

 We first define the geodesic flow on $T\mathbb{S}^2$. The tangent space $T\mathbb{S}^2$ inherits a Riemannian metric as follows (see also \cite[Definition 1.9.12]{Kli95}): for any $X=(q,v)\in T\mathbb{S}^2$, we can split $T_X T\mathbb{S}^2$ into the horizontal subspace $T_{Xh}T\mathbb{S}^2$ and  the vertical subspace $T_{Xv}T\mathbb{S}^2$,
  both are canonically identified with $T_q\mathbb{S}^2$, then we define the Riemannian metric on $T_XT\mathbb{S}^2$ by letting the horizontal space and the vertical space be orthogonal and taking on each of these spaces the Riemannian metric by the canonical identification.
 Let $\omega$ be the standard symplectic form on the tangent bundle $T \mathbb{S}^2$, that is defined by
\begin{equation}\label{eq: symplectic form}
  \omega(\xi,\eta)=g(\xi_h,\eta_v)-g(\eta_h,\xi_v),
\end{equation}
 where  $\xi=(\xi_h,\xi_v)$ and $\eta=(\eta_h,\eta_v)$ are the decompositions such that $\xi_h, \eta_h\in T_q\mathbb{S}^2\cong T_{Xh}T\mathbb{S}^2$ and $\xi_v,\eta_v\in T_{q}\mathbb{S}^2\cong T_{Xv}T\mathbb{S}^2$ (see \cite[Proposition 3.1.14]{Kli95} for details).
  We consider the energy
\begin{eqnarray*}
 H: & T \mathbb{S}^2  &\rightarrow \mathbb{R},\\
   & X  &\mapsto  \frac{1}{2} g(X,X),
\end{eqnarray*}
and the Hamiltonian vector field $\zeta_H\in TT \mathbb{S}^2$ that is defined by
\begin{equation}\label{eq: hamiltonian vector field}
 DH=\omega(\zeta_H,\cdot).
 \end{equation}
Then its flow $\phi^t: T\mathbb{S}^2\to T\mathbb{S}^2$ is called the geodesic flow on $T \mathbb{S}^2$.  The projection of the flow line $\phi^t(X)$ for $X\in TM$ to $S^2$ is the geodesic $\gamma$ on $\mathbb{S}^2$ determined by $\dot{\gamma}(0)=X$ and $\dot{\gamma}(t)=\phi^t(X)$ for all $t$ (cf. \cite[Proposition~3.1.13]{Kli95}).

Now, we will verify that $h_{**}(\zeta_H|_{X})=\zeta_H|_{h_*X}$ for every $X\in T\mathbb{S}^2$, and hence $h_*$ commutes with  the geodesic flow $\phi^t$ on $T \mathbb{S}^2$, i.e.,
\begin{equation}\label{eq: comutative}
  h_*\circ\phi^t=\phi^t\circ h_*,
\end{equation}
where $h_*: T\mathbb{S}^2 \to T\mathbb{S}^2$ is the pushforward  of the isometry $h: \mathbb{S}^2\to \mathbb{S}^2$, and $h_{**}: TT\mathbb{S}^2\to TT\mathbb{S}^2$ is the pushforward of $h_*$.
 Since $h$ is an isometry on $(\mathbb{S}^2,g)$, we have that $H(h_* X)=H(X)$  for any $X=(q,v)\in T_q \mathbb{S}^2$. By differentiating it with respect to $X$, we obtain
 \[DH_{h_* X}(h_{**} \eta)=DH_{X}(\eta), \quad \forall \eta\in T_X T\mathbb{S}^2. \]
Then, by (\ref{eq: hamiltonian vector field})
we obtain
\[\omega(\zeta_H|_{h_*X}, h_{**}\eta)=\omega(\zeta_H|_X, \eta).\]
On the other hand, since $h_*(q,v)=(h(q),h_* v)$,  we have
\[h_{**}(\eta_h,\eta_v)=(h_*\eta_h, h_* \eta_v) \quad \forall \eta\in T_X T\mathbb{S}^2.\]
 Recall that $h$ is an isometry, by (\ref{eq: symplectic form}) we get
\[\omega(h_{**}(\zeta_H|_X), h_{**}\eta)=\omega(\zeta_H|_X, \eta).\]
 So,
\[
 \omega(h_{**}(\zeta_H|_X), h_{**}\eta)=\omega(\zeta_H|_{h_*X}, h_{**}\eta),\quad  \forall \eta\in T_XT \mathbb{S}^2,
\]
which implies $h_{**}(\zeta_H|_X)=\zeta_H|_{h_*X}$.

It is not hard to prove that $H(\phi^t (X))\equiv H(X)$, for $t\in\mathbb{R}$ (see \cite[Proposition 3.1.16]{Kli95}).
So, we can restrict the geodesic flow $\phi^t$ to the unit tangent bundle $S_g\mathbb{S}^2$ of $(\mathbb{S}^2, g)$.
We will give a global surface of section $\Sigma$ for the restricted geodesic flow in the following paragraph.

We know that there is at least one non-contractible closed geodesic on $(\mathbb{R}P^2, g)$ (cf. \cite[Theorem 2.4.19]{Kli95}).
Let $c$ be the non-contractible closed geodesic on $(\mathbb{R}P^2,g)$ which has minimal length among all the non-contractible curves on $(\mathbb{R}P^2,g)$.
Then, the lift of $c$ to  $(\mathbb{S}^2,g)$ is a closed geodesic. We denote it by $\varsigma=c^2$. Moreover, $\varsigma$ is a simple closed geodesic on $(\mathbb{S}^2,g)$ that separates $(\mathbb{S}^2,g)$ into two disks.
Otherwise, $\varsigma$ will have self-intersection and then the projection of $\varsigma$ onto $(\mathbb{R}P^2,g)$ will not be  length-minimum among all the non-contractible curves on $(\mathbb{R}P^2,g)$.
Suppose that $\varsigma$ is parametrized by arclength and of length $\delta$. Denote by $\mathbb{S}^1_{\delta}=\mathbb{R}/\delta\mathbb{Z}$ the circle of length $\delta$.
Then,
\[ \Sigma :=\{(q, v)\in S_g \mathbb{S}^2 \mid q \in \varsigma(\mathbb{S}^1_{\delta}), 0 \leq \angle (\dot{\varsigma}|_q, v)\leq \pi\} \backsimeq \mathbb{S}^1\times [0,\pi],\]
is an embedded surface in the unit tangent bundle $S_g \mathbb{S}^2$ with boundary $\dot{\varsigma}(\mathbb{S}_{\delta}^{1})\sqcup \dot{\overline{\varsigma}}(\mathbb{S}_{\delta}^{1})$, where $ \angle (\dot{\varsigma}|_q, v)$ is the oriented angle from $\dot{\varsigma}|_q$ to $v$ and $\overline{\varsigma}(t):=\varsigma(-t)$.  The Gaussian curvature of $(\mathbb{S}^2,g)$ is positive by the assumption that $(\mathbb{R}P^2,g)$ has positive Gaussian curvature.
Then by Theorems 3.10.2 and 3.10.4 of \cite{Kli95},  $\Sigma $ is a global surface of section. Furthermore,
\[ h_*(\Sigma) =\{(q, v)\in S_g \mathbb{S}^2 \mid q \in \varsigma(\mathbb{S}_{\delta}^1), -\pi \leq \angle (\dot{\varsigma}|_q, v)\leq 0\} ,\]
and $\partial\Sigma $ is invariant under the involution $h_*$.

We  define a Poincar\'{e} half-return map $\psi: \Sigma \setminus \partial\Sigma \rightarrow h_*(\Sigma \setminus \partial\Sigma )$ by
$\psi(x)=\phi^{\tau(x)}(x)$, where $\tau(x) := \min{\{t > 0 \mid\phi^t(x)\in h_*(\Sigma \setminus \partial\Sigma )\}}$. Note that $\psi(q,v)=(q',v')$ iff the unit speed geodesic  on $\mathbb{S}^2$ staring at $q\in\varsigma$ with the initial tangent vector $v$ meets $\varsigma$ again (the first time)  at $q'\in \varsigma$ with the tangent vector $v'$. By \cite[Theorem 3.10.4]{Kli95}, we know that any geodesic, starting from a point on $\varsigma$ in a non-tangent direction, meets $\varsigma$ again. So, the Poincar\'{e} half-return map $\psi$ is well defined.

\begin{claim}
The first conjugate point of $\varsigma(s_0)$ is $h(\varsigma(s_0))$ for any $s_0\in \mathbb{S}_{\delta}^{1}$.
\end{claim}

In fact, by \cite[Theorem 3.10.4]{Kli95}, $\varsigma$ contains a conjugate point. On the other hand, $\varsigma=c^2$ and $c$ is shortest curve connecting $\varsigma(s_0)$ to $h(\varsigma(s_0))$ by our choice of $c$, then by \cite[Theorem 1.12.13 (ii)]{Kli95},
$c$ has no conjugate point of multiplicity $>0$ in its interior and then the claim follows.

 We can extend $\psi$ continuously to the boundary of $\Sigma $:
 If $x=(q, v)=\dot{\varsigma}(s_0)$, then $\psi(x)$ shall be the tangent vector to $\varsigma$ at the first conjugate point $h(\varsigma(s_0))$ of $\varsigma(s_0)$ along $\varsigma(s_0 + s)$, $s\geq 0$;
 if $x=(q,v)=-\dot{\varsigma}(s_0)$, then $\psi(x)$ shall be the tangent vector to $\varsigma(s_0 - s)$, $s\geq 0$ at the first conjugate point of $\varsigma(s_0)$.
 (In fact, for any $\varepsilon>0$ and  sufficiently small $\theta>0$, a geodesic $\gamma$ which starts from the  point $\varsigma(s_0)$ with the initial vector $v$ such that $0<|\angle(\dot{\varsigma}(s_0), v)|<\theta$, meets $\varsigma$ again at a point $\varsigma(s)$ with $|s-s_1|<\varepsilon$, where $\varsigma(s_1)$ is the first conjugate point of $\varsigma(s_0)$ along $\varsigma$ (cf. \cite[Complement 2.1.13]{Kli95}). By the continuity of the flow $\phi^t$, we can prove that this holds uniformly for $s_0$.)

Now, we consider the homeomorphism  $f=h_*^{-1}\circ \psi$ of the annulus $\Sigma$, with $f=\mathrm{Id}$ on $\partial\Sigma$.
Although $h_*$ is orientation reversing,   $f=h_*^{-1}\circ \psi$  preserves the orientation and is isotopic to the identity.  Similarly to the proof of \cite[Theorem 3.10.2]{Kli95},  we know that $f$  preserves the  area induced by the restricted  canonical symplectic form $\omega$.
By applying Theorem \ref{thm: growth of periodic orbits}, we obtain that $f$ has infinitely many odd periodic orbits in the interior of $\Sigma $ and that the number of these odd periodic orbits with prime-periods not exceeding $n$ grows at least like $n^2$.

In the following, we will prove that any odd periodic point $x$ of $f$ in the interior of $\Sigma $ corresponds
to a non-contractible closed geodesic on $(\mathbb{R}P^2,g)$. In fact, let $f^{2k+1}(x)=x$.
Previously, we only defined the Poincar\'{e} half-return map $\psi$ on $\Sigma\setminus \partial \Sigma$, but we can defined it on $h_*(\Sigma\setminus \partial \Sigma)$  similarly.
Moreover, by  (\ref{eq: comutative}),  we know that $h_*\circ\psi=\psi\circ h_*$.
Recall that $h_*^2=\mathrm{Id}$ and $f=h_*^{-1}\circ\psi$. Then $\psi^{2k+1}(x)=h_*(x)$, and $x$ corresponds to a non-contractible closed geodesics on $(\mathbb{R}P^2,g)$.
The proof of the converse is the same.

Hence, there exist infinitely many distinct non-contractible closed geodesics on $(\mathbb{R}P^2,g)$.
Moreover, the lengths of the geodesics, starting at a point $\varsigma(s)$ with initial vector not tangent to $\varsigma$ and ending at its first intersection point with $\varsigma$,   are uniformly bounded from above (see \cite[Proposition~3.10.3]{Kli95}) and from below ($\ge$ the injectivity radius of $M >0$; cf. \cite[Definition~2.1.9 and Proposition~2.1.10]{Kli95}).
So, the number of non-contractible closed geodesics of length $\leq l$ grows at least like $l^2$.
\end{proof}

Motivated by the results of Franks \cite{Franks92}, Bangert \cite{Bangert1993} and  Hingston \cite{Hingston1993}, it's hopeful to obtain a growth rate of the number of non-contractible closed geodesics on Riemannian $\mathbb{R}P^2$ without the condition that the Gaussian curvature is positive. So, we tend to believe that the following conjecture should hold:

\begin{conjecture}
 Let $\mathbb{R}P^2$ be a real projective plane endowed with a Riemannian metric $g$. Then there exist infinitely many distinct non-contractible closed geodesics on $(\mathbb{R}P^2,g)$. Moreover, the number of non-contractible closed geodesics of length $\leq l$ grows at least like the prime numbers.
\end{conjecture}

For Finsler $\mathbb{R}P^2$  under some natural assumption, we can use Theorem \ref{thm: growth of periodic orbits} and the proof of  Lemma \ref{lm: 1}, i.e., \cite[Theorem 11]{LWY}, to obtain the two or infinity results, and also the growth rate about the non-contractible closed geodesics, i.e., Theorem \ref{thm: application 2}.

Before proving Theorem \ref{thm: application 2}, we also give the idea first. Note that $\mathbb{S}^2$ covers $\mathbb{R}P^2$ twice, and  that the antipodal map $h$ of $\mathbb{S}^2$ is the deck transformation. Note also the unit tangent bundle $S_F\mathbb{S}^2\cong SO(3)$ is double covered by $SU(2)\cong \mathbb{S}^3$.  We will prove that $h$ induces a free action of $\mathbb{Z}_4$ on $\mathbb{S}^3$ denoted by $\widetilde{h}_*$, and that the contact form on $S_F \mathbb{S}^2$ induces a contact form $\alpha$ on $\mathbb{S}^3$, such that the geodesic flow on $S_F\mathbb{S}^2$ corresponds to the Reeb flow on $(\mathbb{S}^3,\alpha)$. Moreover, the contact form $\alpha$ is dynamically convex and satisfies $(\widetilde{h}_*)^* \alpha=\alpha$. Then, by the proof of Lemma \ref{lm: 1}, we get an area preserving annulus homeomorphism that is isotopic to the identity,  whose odd periodic orbits correspond to the non-contractible geodesics of $\mathbb{R}P^2$.
We can finish the proof by applying Theorem \ref{thm: growth of periodic orbits}.

\begin{proof}[Proof of Theorem \ref{thm: application 2}]
We consider the universal $2$-cover of $(\mathbb{R}P^2, F)$ and the induced Finsler metric,  i.e. $(\mathbb{S}^2, F)$.
The antipodal map $h$ of
\[\mathbb{S}^2=\{(q_1,q_2,q_3)\in \mathbb{R}^3| q_1^2+q_2^2+q_3^2=1\}\]
 is the deck transformation.
The unit sphere $\mathbb{S}^2$ is a submanifold of $\mathbb{R}^3$, and $T\mathbb{S}^2$ is a submanifold of $T\mathbb{R}^3\cong \mathbb{R}^3\oplus \mathbb{R}^3$. More precisely,
\[T\mathbb{S}^2=\{ (q,v)| q\in \mathbb{S}^2, v\in T_{q}\mathbb{R}^3\cong \mathbb{R}^3, \langle q, v\rangle=0 \},\]
where $\langle\cdot , \cdot \rangle$ is the standard inner product.
We can identify the unit tangent bundle  $S_F \mathbb{S}^2$  of $\mathbb{S}^2$ under the Finsler metric with the $3$ dimensional rotation group
\[SO(3)=\{[q,u, q\times u]| q,u\in\mathbb{R}^3, \langle q,q\rangle=1, \langle u,u\rangle=1, \langle q,u\rangle=0\}\] by
\[(q, v)\mapsto [q, \frac{v}{\sqrt{\langle v,v\rangle}}, q\times \frac{v}{\sqrt{\langle v, v\rangle}}].\]
So, the pushforward
 \begin{eqnarray*}
 h_* : & T\mathbb{S}^2 &\to T\mathbb{S}^2,\\
    & (q, v) & \mapsto (-q,-v),
 \end{eqnarray*}
  of $h$ induces the following diffeomorphism of $SO(3)$, which still denoted by $h_*$,
 \begin{eqnarray*}
 h_* : &  SO(3) & \to SO(3),\\
    & [q, u, q\times u] & \mapsto [-q, -u, q\times u].
 \end{eqnarray*}
By Lemma \ref{lemma: cover SO(3)},
the unit  $3$-dimensional sphere
\[\mathbb{S}^3=\{(z_1, z_2): z_1, z_2\in\mathbb{C}, |z_1|^2+|z_2|^2=1\}\]
 covers  $SO(3)$ twice by  the covering map
 \begin{eqnarray*}
 \pi: & \mathbb{S}^3 & \to SO(3),\\
  & (z_1, z_2) &\mapsto
           \begin{bmatrix}
           \mathrm{Re}(z_1^2-\bar{z}_2^2) & - \mathrm{Im}(z_1^2+\bar{z}_2^2) & 2  \mathrm{Re}(z_1\bar{z}_2)\\
            \mathrm{Im}(z_1^2-\bar{z}_2^2) &  \mathrm{Re}(z_1^2+\bar{z}_2^2)  & 2 \mathrm{Im} (z_1 \bar{z_2})\\
           -2 \mathrm{Re}(z_1 z_2) & 2  \mathrm{Im}(z_1 z_2) & |z_1|^2- |z_2|^2
           \end{bmatrix}.
 \end{eqnarray*}
It is not hard to verify that
\begin{eqnarray*}
 \widetilde{h}_* &: \mathbb{S}^3 &\to \mathbb{S}^3,\\
  & (z_1,z_2)&\mapsto (iz_1, iz_2),
 \end{eqnarray*}
  is a lift of $h_*$, which induces a free action of $\mathbb{Z}_4$ on $\mathbb{S}^3$.

The  cotangent bundle $T^* \mathbb{S}^2$ has a standard one form $\theta$ represented by $\sum p_i dq_i$, where $(q,p)\in T^* \mathbb{S}^2$.
Since $h$ is an isometry on $(\mathbb{S}^2,F)$, the standard one form $\theta$ is invariant under the involution $h^*$.
Thus there is a one form $\alpha$ on the unit tangent bundle $S_F \mathbb{S}^2$ satisfying $(h_*)^*\alpha=\alpha$,  and $(S_F \mathbb{S}^2, \alpha)$ is a contact manifold.

Let $(x_1, y_1, x_2, y_2)$ be the coordinates in $\mathbb{R}^4\cong \mathbb{C}^2$. The restriction of the Liouville one form
\[\alpha_0=\frac{1}{2}\sum_{j=1}^{2}(y_j\mathrm{d}x_j-x_j\mathrm{d}y_j)\]
on $\mathbb{S}^3$ is a contact form, which still denoted by $\alpha_0$. Moreover, we have
\[(\widetilde{h}_*)^*\alpha_0=\alpha_0.\]

By \cite[Section 4]{HaP1},  $(S_F \mathbb{S}^2, \alpha)$  induces a contact form on $\mathbb{S}^{3}$, which we still denote by $\alpha$, such that
\[\alpha=2 g \alpha_0,\]
where $g:\mathbb{S}^3\to \mathbb{R}^+$ satisfies $g\circ \widetilde{h}_*=g$.
Moreover, the geodesic flow of $(\mathbb{S}^2, F)$ (on $S_F\mathbb{S}^2$) is smoothly conjugate (up to a double covering) to the Reeb flow of $(\mathbb{S}^3,\alpha)$.

Now, we have the contact manifold $(\mathbb{S}^3, \alpha)$ that satisfies
\[(\widetilde{h}_*)^*\alpha=\alpha.\]
By Section 6 of \cite{HaP1} and the assumption that $\left(\frac{\lambda}{1+\lambda}\right)^2<K\le 1$, the contact form $\alpha$ is dynamically convex, which together with Lemma \ref{lm: 1}, implies that there exist two or infinitely many $ \widetilde{h}_*$-symmetric periodic orbits on $(\mathbb{S}^{3}, \alpha)$, whose projections onto $\mathbb{R}P^2$ are   non-contractible closed geodesics.

In the proof of Lemma \ref{lm: 1} (see \cite[Theorem 11]{LWY}), we in fact defined a Poincar\'{e} $\frac{1}{p}$-return map $\psi$ ($\frac{1}{4}$ for this theorem) of the Reeb flow  and got an area preserving annulus homeomorphism $f=\widetilde{h}_*^{-1}\circ \psi$ isotopic to the identity. Moreover, the periodic orbits of  $f$ with odd prime-period  correspond to the non-contractible closed geodesics of $\mathbb{R}P^2$. If  there are  more than two non-contractible closed geodesics on $(\mathbb{R}P^2, F)$, $f$ has at least one periodic orbit with odd prime-period. Then,
by  Theorem  \ref{thm: growth of periodic orbits}, the number of periodic orbits of $f$ with odd prime-period $\le n$ grows at least like $n^2$, and hence
the number of non-contractible closed geodesics of length $\leq l$ grows at least like $l^2$.
\end{proof}


\begin{thebibliography}{10}

\bibitem{BTZ1}
Werner Ballmann, Gudlaugur Thorbergsson, and Wolfgang Ziller.
\newblock Closed geodesics and the fundamental group.
\newblock {\em Duke Math. J.}, 48(3):585--588, 1981.

\bibitem{Bangert1993}
Victor Bangert.
\newblock On the existence of closed geodesics on two-spheres.
\newblock {\em Internat. J. Math.}, 4(1):1--10, 1993.

\bibitem{BH}
Victor Bangert and Nancy Hingston.
\newblock Closed geodesics on manifolds with infinite abelian fundamental
  group.
\newblock {\em J. Differential Geom.}, 19(2):277--282, 1984.

\bibitem{BL2010}
Victor Bangert and Yiming Long.
\newblock The existence of two closed geodesics on every {F}insler 2-sphere.
\newblock {\em Math. Ann.}, 346(2):335--366, 2010.

\bibitem{BCLR}
Fran\c{c}ois B\'{e}guin, Sylvain Crovisier, and Fr\'{e}d\'{e}ric Le~Roux.
\newblock Fixed point sets of isotopies on surfaces.
\newblock {\em J. Eur. Math. Soc. (JEMS)}, 22(6):1971--2046, 2020.

\bibitem{Birkhoff13}
George~D. Birkhoff.
\newblock Proof of {P}oincar\'{e}'s geometric theorem.
\newblock {\em Trans. Amer. Math. Soc.}, 14(1):14--22, 1913.

\bibitem{Birkhoff15}
George~D. Birkhoff.
\newblock The restricted problem of three bodies.
\newblock {\em Rend. Circ. Mat. Palermo}, 39:265--334, 1915.

\bibitem{Birkhoff26}
George~D. Birkhoff.
\newblock An extension of {P}oincar\'{e}'s last geometric theorem.
\newblock {\em Acta Math.}, 47(4):297--311, 1926.

\bibitem{Bro19}
Luitzen Egbertus~Jan Brouwer.
\newblock \"{U}ber die periodischen {T}ransformationen der {K}ugel.
\newblock {\em Math. Ann.}, 80(1):39--41, 1919.

\bibitem{Car13}
Constantin Carathéodory.
\newblock \"{U}ber die {B}egrenzung einfach zusammenh\"{a}ngender {G}ebiete.
\newblock {\em Math. Ann.}, 73(3):323--370, 1913.

\bibitem{CDR22}
Vincent Colin, Pierre Dehornoy, and Ana Rechtman.
\newblock On the existence of supporting broken book decompositions for contact
  forms in dimension 3.
\newblock {\em Invent. Math.}, 2022.

\bibitem{CGHP}
Dan Cristofaro-Gardiner, Michael Hutchings, and Daniel Pomerleano.
\newblock Torsion contact forms in three dimensions have two or infinitely many
  {R}eeb orbits.
\newblock {\em Geom. Topol.}, 23(7):3601--3645, 2019.

\bibitem{Ker36}
B\'{e}la de~Ker\'{e}kj\'{a}rt\'{o}.
\newblock Sur la structure des transformations topologiques des surfaces en
  elles-m\^{e}mes.
\newblock {\em Enseign. Math.}, 35:297--316, 1936.

\bibitem{DLW}
Huagui Duan, Yiming Long, and Wei Wang.
\newblock The enhanced common index jump theorem for symplectic paths and
  non-hyperbolic closed geodesics on {F}insler manifolds.
\newblock {\em Calc. Var. Partial Differential Equations}, 55(6):Art. 145, 28,
  2016.

\bibitem{DLX2015}
Huagui Duan, Yiming Long, and Yuming Xiao.
\newblock Two closed geodesics on {$\mathbb{R}P^{2n+1}$} with a bumpy {F}insler
  metric.
\newblock {\em Calc. Var. Partial Differential Equations}, 54(3):2883--2894,
  2015.

\bibitem{Eil34}
Samuel Eilenberg.
\newblock Sur les transformations périodiques de la surface de sphère.
\newblock {\em Fundamenta Mathematicae}, 22(1):28--41, 1934.

\bibitem{Franks88}
John Franks.
\newblock Generalizations of the {P}oincar\'{e}-{B}irkhoff theorem.
\newblock {\em Ann. of Math. (2)}, 128(1):139--151, 1988.

\bibitem{Franks92}
John Franks.
\newblock Geodesics on {$S^2$} and periodic points of annulus homeomorphisms.
\newblock {\em Invent. Math.}, 108(2):403--418, 1992.

\bibitem{FK}
Urs Frauenfelder and Jungsoo Kang.
\newblock Real holomorphic curves and invariant global surfaces of section.
\newblock {\em Proc. Lond. Math. Soc. (3)}, 112(3):477--511, 2016.

\bibitem{GM1969JDG}
Detlef Gromoll and Wolfgang Meyer.
\newblock Periodic geodesics on compact riemannian manifolds.
\newblock {\em J. Differential Geometry}, 3:493--510, 1969.

\bibitem{HW}
Godfrey~Harold Hardy and Edward~Maitland Wright.
\newblock {\em An introduction to the theory of numbers}.
\newblock Oxford University Press, Oxford, sixth edition, 2008.
\newblock Revised by D. R. Heath-Brown and J. H. Silverman, With a foreword by
  Andrew Wiles.

\bibitem{HaP1}
Adam Harris and Gabriel~P. Paternain.
\newblock Dynamically convex {F}insler metrics and {$J$}-holomorphic embedding
  of asymptotic cylinders.
\newblock {\em Ann. Global Anal. Geom.}, 34(2):115--134, 2008.

\bibitem{Hingston1993}
Nancy Hingston.
\newblock On the growth of the number of closed geodesics on the two-sphere.
\newblock {\em Internat. Math. Res. Notices}, (9):253--262, 1993.

\bibitem{HWZ98}
Helmut H.~W. Hofer, Krzysztof Wysocki, and Eduard~J. Zehnder.
\newblock The dynamics on three-dimensional strictly convex energy surfaces.
\newblock {\em Ann. of Math. (2)}, 148(1):197--289, 1998.

\bibitem{HWZ03}
Helmut H.~W. Hofer, Krzysztof Wysocki, and Eduard~J. Zehnder.
\newblock Finite energy foliations of tight three-spheres and {H}amiltonian
  dynamics.
\newblock {\em Ann. of Math. (2)}, 157(1):125--255, 2003.

\bibitem{Kac47}
Mark Kac.
\newblock On the notion of recurrence in discrete stochastic processes.
\newblock {\em Bull. Amer. Math. Soc.}, 53:1002--1010, 1947.

\bibitem{Kang}
Jungsoo Kang.
\newblock On reversible maps and symmetric periodic points.
\newblock {\em Ergodic Theory Dynam. Systems}, 38(4):1479--1498, 2018.

\bibitem{Katok1973}
Anatole~B. Katok.
\newblock Ergodic properties of degenerate integrable hamiltonian systems.
\newblock {\em Math. USSR-Izv.}, 37:535--571, 1973.

\bibitem{Kim}
Seongchan Kim.
\newblock Symmetric periodic orbits and invariant disk-like global surfaces of
  section on the three-sphere.
\newblock {\em Trans. Amer. Math. Soc.}, 375(6):4107--4151, 2022.

\bibitem{Kli95}
Wilhelm P.~A. Klingenberg.
\newblock {\em Riemannian geometry}, volume~1 of {\em De Gruyter Studies in
  Mathematics}.
\newblock Walter de Gruyter \& Co., Berlin, second edition, 1995.

\bibitem{KLN15}
Andres Koropecki, Patrice Le~Calvez, and Meysam Nassiri.
\newblock Prime ends rotation numbers and periodic points.
\newblock {\em Duke Math. J.}, 164(3):403--472, 2015.

\bibitem{Lecalvez01}
Patrice Le~Calvez.
\newblock Rotation numbers in the infinite annulus.
\newblock {\em Proc. Amer. Math. Soc.}, 129(11):3221--3230, 2001.

\bibitem{Lecalvez05}
Patrice Le~Calvez.
\newblock Une version feuillet\'{e}e \'{e}quivariante du th\'{e}or\`eme de
  translation de {B}rouwer.
\newblock {\em Publ. Math. Inst. Hautes \'{E}tudes Sci.}, (102):1--98, 2005.

\bibitem{Liu}
Hui Liu.
\newblock The {F}adell-{R}abinowitz index and multiplicity of non-contractible
  closed geodesics on {F}insler {$\mathbb{R}P^n$}.
\newblock {\em J. Differential Equations}, 262(3):2540--2553, 2017.

\bibitem{LWY}
Hui Liu, Jian Wang, and Jingzhi Yan.
\newblock Refinements of {F}ranks' theorem and applications in {R}eeb dynamics.
\newblock {\em J. Differential Equations}, 338:372--387, 2022.

\bibitem{LX}
Hui Liu and Yuming Xiao.
\newblock Resonance identity and multiplicity of non-contractible closed
  geodesics on {F}insler {$\mathbb{R}P^n$}.
\newblock {\em Adv. Math.}, 318:158--190, 2017.

\bibitem{Mat82}
John~N. Mather.
\newblock Topological proofs of some purely topological consequences of
  {C}arath\'{e}odory's theory of prime ends.
\newblock In {\em Selected studies: physics-astrophysics, mathematics, history
  of science}, pages 225--255. North-Holland, Amsterdam-New York, 1982.

\bibitem{Neumann77}
Walter~D. Neumann.
\newblock Generalizations of the {P}oincar\'{e} {B}irkhoff fixed point theorem.
\newblock {\em Bull. Austral. Math. Soc.}, 17(3):375--389, 1977.

\bibitem{Poincare}
Henri Poincar\'{e}.
\newblock Sur un th\'{e}or\`{e}me de g\'{e}ometrie.
\newblock {\em Rend. Circ. Mat. Palermo}, 33:375--407, 1912.

\bibitem{Rad2004}
Hans-Bert Rademacher.
\newblock A sphere theorem for non-reversible {F}insler metrics.
\newblock {\em Math. Ann.}, 328(3):373--387, 2004.

\bibitem{RT22}
Hans-Bert Rademacher and Iskander~A. Taimanov.
\newblock The second closed geodesic, the fundamental group, and generic
  {F}insler metrics.
\newblock {\em Math. Z.}, 302(1):629--640, 2022.

\bibitem{Sch20}
Alexsandro Schneider.
\newblock Global surfaces of section for dynamically convex {R}eeb flows on
  lens spaces.
\newblock {\em Trans. Amer. Math. Soc.}, 373(4):2775--2803, 2020.

\bibitem{Sternberg}
Shlomo Sternberg.
\newblock {\em Group theory and physics}.
\newblock Cambridge University Press, Cambridge, 1994.

\bibitem{Tai2016}
Iskander~A. Taimanov.
\newblock The spaces of non-contractible closed curves in compact space forms.
\newblock {\em Mat. Sb.}, 207(10):105--118, 2016.

\bibitem{VS1976}
Micheline Vigu\'{e}-Poirrier and Dennis Sullivan.
\newblock The homology theory of the closed geodesic problem.
\newblock {\em J. Differential Geometry}, 11(4):633--644, 1976.

\bibitem{Wang14}
Jian Wang.
\newblock A generalization of the line translation theorem.
\newblock {\em Trans. Amer. Math. Soc.}, 366(11):5903--5923, 2014.

\bibitem{Wright}
Fred~B. Wright.
\newblock Mean least recurrence time.
\newblock {\em J. London Math. Soc.}, 36:382--384, 1961.

\bibitem{XL2015}
Yuming Xiao and Yiming Long.
\newblock Topological structure of non-contractible loop space and closed
  geodesics on real projective spaces with odd dimensions.
\newblock {\em Adv. Math.}, 279:159--200, 2015.

\bibitem{Ziller1982}
Wolfgang Ziller.
\newblock Geometry of the {K}atok examples.
\newblock {\em Ergodic Theory Dynam. Systems}, 3(1):135--157, 1983.

\end{thebibliography}

\end{document}